\documentclass[12pt]{amsart}

\usepackage{amssymb}
\usepackage{fullpage}
\usepackage[all]{xy}
\usepackage{graphicx}

\usepackage[usenames]{color}

\usepackage{epic}
\usepackage{eepic}

\title[A polynomial bound on solutions of quadratic equations]{
A polynomial bound on solutions of quadratic equations in free groups}

\author[]{Igor Lysenok}
\address{Steklov Mathematical Institute, Moscow, Russia}
\email{igor.lysenok@gmail.com}
\thanks{
The first author has been supported by the Russian Foundation for Fundamental Research
}

\author[]{Alexei Myasnikov}
\address{Department of Mathematics, Stevens Institute of Technology,
Hoboken, USA} 
\email{amiasnikov@gmail.com}

\newtheorem{theorem}{Theorem}[section]
\newtheorem{lemma}[theorem]{Lemma}
\newtheorem{proposition}[theorem]{Proposition}

\newtheorem{corollary}[theorem]{Corollary}

\newtheorem*{maincorollary}{Corollary \ref{cor:standard-reduction-general}}

\theoremstyle{definition}
\newtheorem{definition}[theorem]{Definition}
\newtheorem{example}[theorem]{Example}

\theoremstyle{remark}

\DeclareMathOperator\Aut{Aut}
\DeclareMathOperator\End{End}
\DeclareMathOperator\Var{Var}
\DeclareMathOperator\Stab{Stab}
\DeclareMathOperator\Hom{Hom}
\DeclareMathOperator\Sol{Sol}
\DeclareMathOperator\im{Im}

\newcommand\set[1]{\{ #1 \}}
\newcommand\setof[2]{\{ #1 \;|\; #2 \}}

\let\al\alpha
\let\be\beta
\let\ga\gamma
\let\de\delta
\let\ep\varepsilon
\let\la\lambda

\let\th\theta
\let\si\sigma
\let\om\omega
\let\Ga\Gamma
\let\De\Delta
\let\Th\Theta

\newcommand\Reals{{\mathbb R}}

\let\eset\varnothing
\let\bd\partial
\let\ti\tilde
\let\sm\setminus

\def\theenumi{(\roman{enumi})}  

\catcode`@ = 11
\def\p@enumi{\thelemma}

\let\@savedlabel\label
\def\label#1{\@savedlabel{#1}\ifnum\@listdepth=1%
\protected@edef\@currentlabel{\theenumi}\@savedlabel{#1-it@m}\fi}

\begin{document}

\begin{abstract}
We provide polynomial upper bounds on the size of a shortest solution for quadratic equations in a free group.
A similar bound is given for parametric solutions in the description of solutions sets of quadratic equations 
in a free group.
\end{abstract}

\maketitle

\section{Introduction} \label{sec:intro}

Let $G$ be a group and $X$ a countable set of {\em variables}. An {\em equation in $G$}
is an element $E$ of the free product $G \ast F_X$ where $F_X$ is the free group freely generated by ~$X$. 
Usually we write equations in the classical form  $E = 1$
where $E$ is a word over $G \cup X^{\pm1}$ representing a reduced product in $G \ast F_X$.
Elements of $G$ occurring in $E$ are called {\em coefficients} of an equation $E=1$.
A {\em solution} of $E = 1$ in $G$  is a homomorphism 
$\alpha: G \ast F_{X} \to G$  such that $g^\alpha  = g$ for every $g \in G$ 
(so-called {\em $G$-homomorphism}) and $E^\alpha = 1$.  
Sometimes it is convenient to consider the quotient $G_E = G \ast F_X / (E = 1)$
of $G \ast F_X$ 
modulo the normal subgroup generated by~$E$, called the {\em equation group of} $E= 1$. 
In this case solutions of $E = 1$ in $G$ are precisely 
$G$-homomorphisms $G \ast F_X \to G$. (Observe that if $E = 1$ has a solution
then $G$ embeds into ~$G_E$.)   
In general, there are two natural
problems concerning equations from  a given class ~${\mathcal E}$. The first one
is  the famous Diophantine problem: does   there exist an algorithm to check
whether or not  a given equation from ${\mathcal E}$ has a solution in a group
$G$. The second problem is to get an 
effective
description of the set of
solutions of an equation in ~${\mathcal E}$.

A word $E$ and an equation $E = 1$ are termed  {\em quadratic} if every variable $x \in X$
occurring in $E$ (as $x$ or $x^{-1}$) occurs precisely twice.  
Quadratic equations form a very special class among equations of a general form in groups.
However, they play 
an important role in several areas of mathematics.  
This is not very surprising  since
 quadratic equations are naturally related to the topology of 
compact
surfaces (see for example  \cite{Culler-1981,Olshanskii-1989} and Section \ref{sec:surfaces} below).
Quadratic equations 
groups
naturally 
appear
in  JSJ-decompositions of groups (via
QH subgroups) and  
as group actions of dynamical systems  (via interval
exchange). 
Moreover, quadratic equations play a fundamental part in
general theory of equations in groups and algebraic (Diophantine) geometry over
groups \cite{BMR-1998,KM-IrredI,KM-IrredII,KM-ift,Sela-2001-2006}. 
It has been shown by Kharlampovich and Myasnikov \cite{KM-IrredI,KM-IrredII}
that every system of equations in a free group is rationally equivalent to finitely many
systems in non-degenerate triangular quasi-quadratic form (NTQ), which gives  a
precise analog of the elimination theory from the classical algebraic geometry,
with proper extension theorems and nice algorithmic properties (for details see
surveys \cite{KM-survey1,KM-survey2}).  NTQ systems are crucial in Implicit
Function Theorems in free groups \cite{KM-ift} and Tarski problems \cite{KM-Tarski,Sela-2001-2006}.

The theory  of quadratic equations in a free group $F$ goes back to the works
of Lyndon~\cite{Lyndon-1959}  and  Malcev \cite{Malcev1962}, who described,
correspondingly,   all solutions of the equation $x^2y^2z^2 = 1$  and $[x,y] =
[a,b]$ (here $x,y,z \in X$ are variables and $a, b \in F$ are the generators of $F$). 
Malcev anticipated several modern techniques,  
introducing automorphic equivalence of solutions and focusing on minimal solutions in each automorphic orbit. 
In 
\cite{Wicks-1972}
 Wicks gave a decision
algorithm for the Diophantine problem for   equations of the type 
$W = g$ 
$(g \in F)$. He showed 
that
the problem of
solving 
such
an equation can be effectively reduced to solving finitely many particular
equations in a free monoid with involution (via ``Wicks forms'').  

 In \cite{ComerfordEdmunds-1981}, Comerford and Edmunds proved that the
Diophantine problem for arbitrary quadratic equations in a free group $F$ is
decidable. 
  Later  Comerford and Edmunds ~\cite{ComerfordEdmunds-1989} and
Grigorchuk and Kurchanov \cite{Grigorchuk-Kurchanov-1989}  completely described
solution sets of quadratic equations in free groups (the result will
be formulated below in this section).
  
  In 1982,  Makanin \cite{Makanin-1982}
proved that
 the Diophantine problem in free groups
is solvable,
and a few years later Razborov \cite{Razborov-1987} gave a description of
solutions of systems of equations in free groups. These two
very influential papers shaped the modern theory of equations in groups. 
Note that the description of solutions sets of quadratic equations given in 
\cite{ComerfordEdmunds-1989,Grigorchuk-Kurchanov-1989} is a very special case of 
Razborov's description for general equations.
 
 Techniques for solving equations in free and related groups were instrumental
in solution of some other decision problems in group theory: for example, the
isomorphism problem in hyperbolic groups \cite{Sela-1995, Dahmani-Guirardel-2010}, limit groups
\cite{KM-iso}, and toral relatively hyperbolic groups \cite{Dahmani-Groves-2008}. 
 
 In view of applications, the principal question concerning equations in groups
is the time complexity of decision algorithms. It has been shown by
Bormotov, Gilman, and Myasnikov ~\cite{BGM-2009} that one-variable equations in a free group
admit polynomial time decision algorithms.
Ol'shaskni\u\i\  \cite{Olshanskii-1989} 
and  Grigorchuk and Kurchanov \cite{Grigorchuk-Kurchanov-1989}
proved that if  the number of variables 
is fixed then the Diophantine problem for quadratic equations in free groups 
has a decision algorithm  polynomial in the sum of the lengths of the
coefficients. However, this is as far as one can go in polynomial time. 
First, Diekert and Robson showed in \cite{Diekert-Robson-1999} that the Diophantine
problem for quadratic equations in free monoids (semigroups) is NP-hard. Then
Kharlampovich, Lysenok, Myasnikov and Touikan proved that the Diophantine
problem for quadratic equations in free groups is precisely NP-complete \cite{KLMT}. 
Nevertheless, a few important questions remained to be open,  above all, 
the question whether decidable quadratic equations in free groups have solutions
polynomially bounded in the size of the equation.

The affirmative answer to this question is given by the following theorem. Here a quadratic
word $Q$ is {\em orientable} if every variable $x$ occurring in $Q$ occurs in $Q$ with two 
opposite exponents as $x$ and $x^{-1}$ and {\em non-orientable} otherwise, that is, if a variable
occurs in $Q$ twice with the same exponent $+1$ or $-1$.

\begin{theorem} \label{thm:sol-bound}
Let $Q$ be a quadratic word. If the equation $Q=1$ is solvable in a free group ~$F_A$ 
then there exists a solution $\al$ such that for any variable $x$, 
$$
  |x^\al| \le 
  \begin{cases}
    N n(Q) c(Q) & \text{if $Q$ is orientable}, \\
    N n(Q)^2 c(Q) & \text{if $Q$ is non-orientable}, \\
  \end{cases}
$$
for some constant $N$. Here $n(Q)$ denotes the total number of variables in $Q$ and $c(Q)$ the 
total length of coefficients occurring in $Q$. 
One can take $N = 
40
$ for orientable $Q$ and $N=
150
$ for non-orientable ~$Q$.

If $Q$ is standard orientable or semi-standard non-orientable 
(see Definitions \ref{def:standard} and \ref{def:semi-standard})
then there exists $\al$ with a better bound
$$
  |x^\al| \le 
  \begin{cases}
    N c(Q) & \text{if $Q$ is orientable}, \\
    N n(Q) c(Q) & \text{if $Q$ is non-orientable}, \\
  \end{cases}
$$
with $N=
8
$ and $N=
36
$ respectively.
\end{theorem}

In a similar manner, we give a bound on the size of 
parametric
solutions of quadratic equations which participate in 
the description of their solution sets.
To state the result we need to define a concept of a parametric solution. Let $T$ be 
an alphabet of {\em parameters} which is assumed to be disjoint 
from the alphabet of constants ~$A$ and alphabet of variables $X$. 
A {\em parametric solution} of an 
equation $E=1$ in $F_A$ is an $F_A$-homomorphism $\eta: F_{A\cup \Var(E)} \to F_{A \cup T}$ such that $E^\eta=1$
where $\Var(E)$ denotes the set of variables occurring in $E$. If $\eta$ is a parametric solution of $E=1$ then
for any $F_A$-homomorphism $\psi: F_{A\cup T} \to F_A$ we get a solution $\eta\psi$ of $E=1$ in the usual
sense, a {\em specialization} of $\eta$.
Note that we use here a notion of a solution of an equation $E=1$ 
in a free group ~$F_A$ which is slightly different from one introduced above.
Instead of taking $F_A$-homomorphisms of $F_{A \cup X}$ we restrict them
to the free group $F_{A \cup \Var(E)}$ 
involving only 
variables which occur in $E$. 
This provides a more convenient way for describing 
{\em all} possible solutions of a given equation $E=1$.

Let $\Stab(E)$ denote the group of all $F_A$-automorphisms $\phi$ of $F_{A\cup \Var(E)}$ 
such that $E^\phi$ is conjugate to $E$.
This group acts on the solution set of $E=1$ by left multiplications.
Hence any parametric solution $\eta$ produces a whole bunch of solutions 
in the usual sense, 
the union of orbits of specializations of $\eta$:
$$
  \Sol(\eta) = \setof{\phi\eta\om}{\phi \in \Stab(E), \ \om \in \Hom_{F_A}(F_{A\cup T}, F_A)}.
$$

The above mentioned result of Comerford--Edmunds \cite{ComerfordEdmunds-1989} and 
Grigorchuk--Kurchanov \cite{Grigorchuk-Kurchanov-1989} asserts
that for any quadratic equation $E=1$ in $F_A$ there is (and can be effectively produced) a finite set
$\set{\eta_i}$ of {\em basic} parametric solutions such that the set of all solution of $E=1$ is the union
$\cup_i \Sol(\eta_i)$.

Let $\eta$ and $\th$ be two parametric solutions of the same equation $E=1$ in $F_A$. Let us say
that $\eta$ is a {\em generalization} of $\th$ 
if there are an automorphism $\phi\in \Stab (E)$ and an endomorphism 
$\om\in \End_{F_A} (F_{A\cup T})$ such that $\th = \phi\eta\om$. 
Clearly, in this case we have $\Sol(\eta) \supseteq \Sol(\th)$.

\begin{theorem} \label{thm:param-sol-bound}
Let $Q=1$ be a quadratic equation in a free group $F_A$. 
Then any parametric solution of $Q = 1$ has a generalization $\eta$ such that for any variable $x$, 
the length of 
$x^\eta$ is bounded by the same function as in Theorem \ref{thm:sol-bound} with $c(Q)$ replaced by
$c(Q)+ 
2
n(Q)$.

In particular, there is a finite set of basic parametric solutions of $Q=1$ satisfying this bound.
\end{theorem}

Note that a 
simple
description of solution sets (in terms of basic 
parametric
solutions) is known for coefficient-free
quadratic equations. In this case, only one basic solution is enough.
Moreover, if an equation is in the standard form 
then the value of each variable in the basic parametric solution
is either a parameter letter or trivial, see
\cite[Section 5]{Grigorchuk-Kurchanov-1993}. This basic parametric solution is a generalization of every
other
parametric solution, so the theorem does not give much in this case.

The proof of Theorems \ref{thm:sol-bound} and \ref{thm:param-sol-bound}
has three main constituents. In Section \ref{sec:quadratic-transformations} we
produce 
various 
automorphisms 
that allow
to transform quadratic words to a desired form.
In particular, we re-prove a well known fact (%
closely related to
the classification theorem for
compact simplicial surfaces, see for example \cite[Chapter VI]{Seifert-Threllfal}) 
that a quadratic word can be equivalently transformed 
to a word belonging to one of the four standard series,
see
Proposition \ref{prop:quadratic-standard}.
However, in our proof, 
we provide an ``economical'' transformation
(Propositions \ref{prop:standard-red-bound-orientable}, \ref{prop:semistandard-red-bound-nonorientable},
\ref{prop:standard-red-bound-nonorientable}, \ref{prop:semistandard-red-bound-inverse} and
\ref{prop:standard-red-bound-nonorientable-inverse}.)
 As a corollary we formulate a general 
result on the 
bound on the complexity of automorphisms which reduce a given quadratic word to the standard form 
over {\em any\/} group ~$G$.

\begin{maincorollary} 
Let $Q \in G * F_X$ be quadratic word over 
an arbitrary 
group $G$. Then
there are automorphisms $\phi,\psi \in \Aut_G(G * F_X)$
such that $Q^\phi$ and $Q^\psi$ are conjugate to the standard quadratic words equivalent to $Q$ and for any variable $x$,
$$
  |x^\phi| , |x^{\psi^{-1}}| \le
    \begin{cases}
    4 n(Q) + 2 c(Q) & \quad \text{if $Q$ is orientable}, \\
    8 n^2 (Q) + 4 n(Q) c(Q)  & \quad \text{if $Q$ is non-orientable},
    \end{cases}
$$
where $c(Q)$ is the total length of coefficients of $Q$ expressed in any left-invariant (e.g.\ word) 
metric on $G$.
\end{maincorollary}

Note that we provide here bounds for two different automorphisms: 
for the direct transformation $Q \xrightarrow\phi R$ and for the inverse transformation 
$R \xrightarrow{\psi^{-1}} Q$ where $R$ denotes the standard form of $Q$. The reason is that 
a bound on a automorphism $\phi$ of a free group $F$ does not 
imply a reasonable bound on its inverse $\phi^{-1}$;
in particular, see Example \ref{exmpl:inverse-streching} below.

As the next step of our argument we  develop a version of the elimination process 
for quadratic equations in a free group. To do this, we define a certain set
of transformations on pairs of the form $(Q,\al)$ where $Q$ is a coefficient-free quadratic word
and $\al$ is an evaluation of variables in $Q$, i.e.\ an $F_A$-homomorphism $F_{A \cup \Var(Q)} \to F_A$. 
Then we define a certain transformation 
sequence starting from a given pair $(Q,\al)$
whose primary goal is to eliminate cancellations in the formal word $Q[\al]$
obtained by substituting in $Q$ of the value $x^\al$ of each variable ~$x$.
Our approach is essentially an improved version of similar approaches
used in \cite{ComerfordEdmunds-1981,ComerfordEdmunds-1989,Grigorchuk-Kurchanov-1989} for
description of solution sets of quadratic equations in a free group. 
We start in Section \ref{sec:simple-elimination} with a lighter version of it,
with no care about transformation homomorphisms, and prove Proposition \ref{prop:main-reduction} 
sufficient
for the proof of Theorem \ref{thm:sol-bound}. Our full version of the eliminated process
is elaborated in Section \ref{sec:parametric} for the proof of Theorem \ref{thm:param-sol-bound}.

Note that there is an alternating (and essentially equivalent) approach to solutions of quadratic
equations in free groups using Lyndon--van Kampen diagrams on surfaces, see for example 
\cite{Olshanskii-1989}. 
Though this approach provides a clear geometric vision,
an advantage of using elimination process is an easier control of transformation homomorphisms.

The final ingredient to the proof is a statement about Lyndon--van Kampen diagrams
(Proposition \ref{prop:diagram-unfolding}) which says that a diagram can be unfolded in an economical way.

\section{Quadratic words and surfaces} \label{sec:surfaces}

In this section, we recall some elementary concepts and formulate some known facts about quadratic words.
The  free group $F_A$ and a countable set $X$ of variables will be fixed throughout the
whole paper. 
The following fact is well known (see for example \cite[Section 4]{Grigorchuk-Kurchanov-1993}).

\begin{proposition} \label{prop:quadratic-standard}
By an $F_A$-automorphism of $F_{A\cup X}$, every quadratic word can be reduced to one of the following
forms: 
\begin{gather*}
  [x_1,y_1] [x_2,y_2] \dots [x_g,y_g] \quad (g \ge0), \\
  x_1^2 x_2^2 \dots x_g^2 \quad (g >0), \\
  [x_1,y_1] [x_2,y_2] \dots [x_g,y_g] \, c_1 z_2^{-1} c_2 z_2 \dots z_{m}^{-1} c_{m} z_{m} 
    \quad (g \ge0, \ m\ge 1), \\
  x_1^2 x_2^2 \dots x_g^2 \, c_1 z_2^{-1} c_2 z_2 \dots z_{m}^{-1} c_{m} z_{m} \quad (g >0, \ m\ge 1).
\end{gather*}
Here $x_i$, $y_i$ and $z_i$ are variables and $c_i\in F_A$, $c_i\ne 1$, are coefficients.
\end{proposition}

\begin{definition} \label{def:standard}
A quadratic word $Q$ and an equation $Q=1$ are called {\em standard} if $Q$ belongs to one of the
four series given in Proposition \ref{prop:quadratic-standard}.
\end{definition}

The proof of the proposition essentially repeats the proof of the classification theorem for compact 
surfaces, see for example a classical topology textbook \cite[Sections 38--40]{Seifert-Threllfal}. 
In the next section we refine the proof 
and formulate a series of propositions
that the reduction automorphism can be chosen economically in terms of its size.

The series of quadratic words in Proposition \ref{prop:quadratic-standard} represent the four series of compact
surfaces: closed orientable surfaces of genus ~$g$, closed non-orientable surfaces of genus ~$g$, 
orientable surfaces of genus ~$g$ with ~$m$ boundary components and non-orientable surfaces of genus ~$g$
with ~$m$ boundary components, respectively. 

The relation between quadratic words and surfaces relies upon the following simple observation.
Let ~$Q$ be a quadratic word.
Take a 2-disk $D$ with boundary $\bd D$ divided into ~$|Q|$ arcs. 
Choose an orientation of the boundary $\bd D$ and label the arcs with letters 
of $Q$ in the order they appear in $Q$ 
(instead of $x^{-1}$ we put $x$ and direct the arc in the opposite to the orientation of $\bd D$). 
Then for each variable $x$ occurring in $Q$, glue together the two arcs labelled by the two 
occurrences of $x$ in $Q$, according to their orientation. We get a 2-complex ~$S_Q$ representing 
a compact surface. The surface is closed if $Q$ is coefficient-free and has a boundary otherwise. 
(See Fig. ~\ref{fig:surface} where $Q=x^{-1} a y^{-1} b x c y d$ and $S_Q$ is 
homeomorphic to a torus with a disk removed.)
\begin{figure}[h]%
\begin{picture}(0,0)%
\includegraphics{fig1.pstex}%
\end{picture}%
\setlength{\unitlength}{0.00087489in}%
\begingroup\makeatletter\ifx\SetFigFont\undefined%
\gdef\SetFigFont#1#2#3#4#5{%
  \reset@font\fontsize{#1}{#2pt}%
  \fontfamily{#3}\fontseries{#4}\fontshape{#5}%
  \selectfont}%
\fi\endgroup%
{\renewcommand{\dashlinestretch}{30}%
\begin{picture}(5368,1867)(0,-10)
\put(2270,931){\makebox(0,0)[lb]{\smash{{\SetFigFont{12}{14.4}{\rmdefault}{\mddefault}{\updefault}$\longrightarrow$}}}}
\put(867,859){\makebox(0,0)[lb]{\smash{{\SetFigFont{12}{14.4}{\rmdefault}{\mddefault}{\updefault}$D$}}}}
\put(3132,77){\makebox(0,0)[lb]{\smash{{\SetFigFont{12}{14.4}{\rmdefault}{\mddefault}{\updefault}$S_Q$}}}}
\put(3855,607){\makebox(0,0)[lb]{\smash{{\SetFigFont{12}{14.4}{\rmdefault}{\mddefault}{\updefault}$b$}}}}
\put(4160,612){\makebox(0,0)[lb]{\smash{{\SetFigFont{12}{14.4}{\rmdefault}{\mddefault}{\updefault}$c$}}}}
\put(3842,247){\makebox(0,0)[lb]{\smash{{\SetFigFont{12}{14.4}{\rmdefault}{\mddefault}{\updefault}$a$}}}}
\put(4160,250){\makebox(0,0)[lb]{\smash{{\SetFigFont{12}{14.4}{\rmdefault}{\mddefault}{\updefault}$d$}}}}
\put(4372,225){\makebox(0,0)[lb]{\smash{{\SetFigFont{12}{14.4}{\rmdefault}{\mddefault}{\updefault}$x$}}}}
\put(5075,890){\makebox(0,0)[lb]{\smash{{\SetFigFont{12}{14.4}{\rmdefault}{\mddefault}{\updefault}$y$}}}}
\put(1468,255){\makebox(0,0)[lb]{\smash{{\SetFigFont{12}{14.4}{\rmdefault}{\mddefault}{\updefault}$a$}}}}
\put(1730,933){\makebox(0,0)[lb]{\smash{{\SetFigFont{12}{14.4}{\rmdefault}{\mddefault}{\updefault}$y$}}}}
\put(1488,1504){\makebox(0,0)[lb]{\smash{{\SetFigFont{12}{14.4}{\rmdefault}{\mddefault}{\updefault}$b$}}}}
\put(862,1747){\makebox(0,0)[lb]{\smash{{\SetFigFont{12}{14.4}{\rmdefault}{\mddefault}{\updefault}$x$}}}}
\put(287,1503){\makebox(0,0)[lb]{\smash{{\SetFigFont{12}{14.4}{\rmdefault}{\mddefault}{\updefault}$c$}}}}
\put(15,925){\makebox(0,0)[lb]{\smash{{\SetFigFont{12}{14.4}{\rmdefault}{\mddefault}{\updefault}$y$}}}}
\put(880,15){\makebox(0,0)[lb]{\smash{{\SetFigFont{12}{14.4}{\rmdefault}{\mddefault}{\updefault}$x$}}}}
\put(207,280){\makebox(0,0)[lb]{\smash{{\SetFigFont{12}{14.4}{\rmdefault}{\mddefault}{\updefault}$d$}}}}
\end{picture}
}
\caption{}\label{fig:surface}
\end{figure}

Directed edges of 2-complex $S_Q$ carry in a natural way labels in $A^{\pm1} \cup X^{\pm1}$. The edges in the interior
of the surface are labelled by variables in $\Var(Q)$ and 
edges in the surface boundary labelled by letters in $A^{\pm1}$ which come from coefficients of $Q$. 
In case of orientable ~$S_Q$ we fix a positive direction of 
passing along each boundary component
which agrees with the chosen orientation of
the boundary of disk $D$. 
In this case, we view the label of the boundary component as an element of 
$F_A$ defined up to conjugacy.
If $S_Q$ is non-orientable then labels of boundary 
components are viewed as elements of $F_A$ 
defined up to conjugacy and taking the inverse.

The topological type of $S_Q$ may be different from the topological type of $S_R$ where $R$ is a standard
form of $Q$ from Proposition \ref{prop:quadratic-standard}. 
It is not hard to see that if 
the label $c$ of a boundary component of ~$S_Q$ represents the
trivial element of $F_A$ 
then the corresponding boundary component disappears in $S_R$. 
We introduce an extra reduction step for $S_Q$:
if $c$ is trivial, attach a 2-disk along the corresponding boundary 
component of $S_Q$. We denote $\bar{S}_Q$ the resulting surface.

\begin{definition}
Let $Q$ be a quadratic word. We call the unordered tuple of all nontrivial labels of boundary components 
of $S_Q$, viewed as elements of $F_A$,
 the {\em standard coefficients} of $Q$. As remarked above, the standard coefficients are defined
up to conjugacy if $S_Q$ is orientable and up to conjugacy and taking inverses if $S_Q$ is non-orientable.
\end{definition}

\begin{definition} \label{def:quadratic-equivalent}
Two quadratic words $Q$ and $R$ are {\em equivalent} if $\bar{S}_{Q}$ and $\bar{S}_{R}$ have the same 
topological type and their tuples 
$(c_{11},\dots,c_{m1})$ and $(c_{12},\dots,c_{m2})$ of standard coefficients 
are the same up the the natural equivalence:
if $S_{Q}$ is orientable then up to enumeration, each ~$c_{i1}$ is
conjugate to the corresponding $c_{i2}$; if $S_{Q}$ is non-orientable then up to enumeration, each
~$c_{i1}$ is conjugate to either $c_{i2}$ or $c_{i2}^{-1}$.
\end{definition}

With a slight abuse of the language, we call a quadratic word $Q$ itself {\em orientable} if $S_Q$ is an orientable
surface and {\em non-orientable} otherwise. It is immediate from the construction of ~$S_Q$ that $Q$
is orientable if and only if each variable in $Q$ has two occurrences with different 
exponents
$-1$ and $+1$.
By the {\em genus} of $Q$ we mean the genus of $S_Q$. 

Thus, any quadratic word $Q$ is defined up to equivalence by the following parameters: orientability (true/false),
the genus $g$ and the tuple of standard coefficients $(c_1,\dots,c_{m})$.
Proposition \ref{prop:quadratic-standard} asserts precisely that any quadratic word can be reduced
by an $F_A$-automorphism of $F_{A\cup X}$ to an equivalent standard quadratic word.

The following fact may be attributed to folklore. (We don't make use of it in the paper and provide only
a sketch of the proof.)

\begin{proposition}
Let $Q,R \in F_{A\cup X}$ be two quadratic words.
Then the following statements are equivalent:
\begin{enumerate}
\item 
$R$ is equivalent to $Q$;
\item
$R$ is the image of $Q$ under an $F_A$-automorphism of $F_{A\cup X}$;
\item
$R$ is conjugate to the image of $Q$ under an $F_A$-automorphism of $F_{A\cup X}$.
\end{enumerate}
\end{proposition}

\begin{proof}[Sketch of the proof:]
To prove implication (i)$\Rightarrow$(ii), 
we reduce two given equivalent quadratic words to equivalent
standard forms by Proposition \ref{prop:quadratic-standard}. It is then an easy exercise to show
that two equivalent standard quadratic words are images of each other under an $F_A$-automorphisms of
$F_{A\cup X}$ (one can use also automorphisms from Lemmas \ref{lm:permut-coefficients} and
\ref{lm:semistandard-transition} below).

Implication (iii)$\Rightarrow$(i)
follows easily from the following statement: If $Q, R \in F_{A \cup X}$
are quadratic words and 
$R$ is conjugate to $Q^\phi$ 
for some $\phi\in \Aut_{F_A} (F_{A\cup X})$ then $\phi$ may be represented as a product
$\tau_1 \tau_2 \dots \tau_k$ of elementary Nielsen $F_A$-automorphisms ~$\tau_i$ such that 
the cyclically reduced form of 
$Q^{\tau_1\dots \tau_{i}}$ is a quadratic word for each $i$.
To prove the statement, 
we use the Higgins--Lyndon approach to stabilizers in $\Aut(F)$ as exposed in Lyndon and Schupp's
book \cite[Section I.4]{Lyndon-Schupp}. 
We apply a modified version of Proposition I.4.23 of \cite{Lyndon-Schupp} to 
one cyclic word $u_1=Q$ and the tuple of non-cyclic words $u_i$, $i\ge 2$, consisting of 
all letters $a\in A$.
By this proposition, $\phi$ can be represented as a product 
$$
  \phi = \rho_1 \rho_2 \dots \rho_r
$$ 
of Whitehead automorphisms ~$\rho_i$ such that for some $p$ and $q$, $0 \le p \le q \le r-1$, 
$$
  |Q_{i+1}| < |Q_i|  \text{ for } i=0,\dots,p, \quad
  |Q_{i+1}| = |Q_i|  \text{ for } i=p+1,\dots,q
$$
and
$$ 
  |Q_{i+1}| > |Q_i|  \text{ for } i=q+1,\dots,r-1
$$ 
where, by definition, $Q_0 = Q$ and 
$Q_i$ is the cyclically reduced form of $Q^{\rho_1\rho_2\dots\rho_i}$. 
Then it is not hard to see that in each of the chains
$$
  Q_0 \xrightarrow{\rho_1} Q_1 \xrightarrow{\rho_2}  \dots  \xrightarrow{\rho_{q}} Q_{q}
    \quad\text{and}\quad
  Q_r \xrightarrow{\rho_r^{-1}} Q_{r-1} \xrightarrow{\rho_{r-1}^{-1}} \dots \xrightarrow{\rho_{q+1}^{-1}} Q_{q}
$$
every automorphism $\rho_i^{\pm1}$ can be factored into a sequence of elementary Nielsen automorphisms which
keep the property of a word being quadratic.
\end{proof}

\section{Transforming quadratic words} \label{sec:quadratic-transformations}

In this section, we describe several specific transformations of quadratic words to produce equivalent
quadratic words of a desired form. By a transformation we mean application of an $F_A$-automorphism of 
$F_{A \cup X}$ (or an $F_C$-automorphism of $F_{C\cup X}$ where $F_C$ is the group of formal
coefficients, see below).

By $|W|$ we denote the length of an element of a free group written as a frely reduced word. 
The notation $|W|_x$ will be 
used for the total number of occurrences of a letter $x$ in $W$. More generally, is $S$ is a set of 
letters then $|W|_S$ will denote the total number of occurrences in $W$ of letters from $S$.
Sometimes we consider formal (not necessarily freely reduced) words $W$. In this case, $|W|$ denotes
the length of $W$.

For automorphisms of a free group $F_Y$, we use a notation
$$
  \phi=(x_1^{\ep_1}\mapsto W_1, \ x_2^{\ep_2}\mapsto W_2, \ \dots\ x_k^{\ep_k}\mapsto W_k)
  \quad \text{where} \ x_i\in Y,\ \ep_i=\pm1,\ W_i \in F_Y
$$
which means that $\phi$ maps generator $x_i^{\ep_i}$
 to $W_i$ and does not change other generators.

We describe several types of elementary transformations 
$\phi$ 
applied to a quadratic word $Q \in F_{A\cup X}$ and producing an equivalent quadratic word $Q^\phi$,
or, in a weaker form, an equivalent quadratic word conjugate to the image $Q^\phi$ of $Q$.

{\em Permutations and 
exponent 
sign changes of variables:} Permutations on the set of variables and automorphisms 
of the form $(x \mapsto x^{-1})$. 
These automorphisms 
always carry quadratic words into equivalent ones. We will implicitly 
assume that automorphisms of this type are applied if needed. For example, to transform a word $Q_1$ 
to a given word $Q_2$ it is enough to find an automorphism $\phi \in \Aut_{F_A}(F_{A \cup X})$ such 
that $Q_1^\phi$ is equal to $Q_2$ up to renaming variables and changing their exponent signs.

{\em 
Multiplications 
by coefficients:} Automorphisms of the form $(x^\ep\mapsto x^\ep d)$, $d \in F_A$.
Any automorphism of this form carries any quadratic word to an equivalent one. 
We will use only the special case when $x^\ep d^{-1}$ or
$d x^{-\ep}$ occurs in $Q$ as a subword. We refer to these  transformations as  {\em coefficient shifts}. 
Geometrically, we shift the start of 
the edge labelled $x^\ep$ along the boundary arc of $S_Q$ labelled $d^{-1}$ from the start of this arc to its end
(Fig.\ \ref{fig:coef-shift1}). 
This automorphism does not change $S_Q$ unless the corresponding
boundary 
component is labelled by the trivial element 
and there are two occurrences of $(x^\ep d^{-1})^{\pm1}$ in $Q$
(Fig.\ \ref{fig:coef-shift2}). In this case, this boundary component disappears
(but the topological type of the reduced surface $\bar{S}_Q$ is not changed). 
\begin{figure}[h]%
\begin{picture}(0,0)%
\includegraphics{fig2.pstex}%
\end{picture}%
\setlength{\unitlength}{0.00087489in}%
\begingroup\makeatletter\ifx\SetFigFont\undefined%
\gdef\SetFigFont#1#2#3#4#5{%
  \reset@font\fontsize{#1}{#2pt}%
  \fontfamily{#3}\fontseries{#4}\fontshape{#5}%
  \selectfont}%
\fi\endgroup%
{\renewcommand{\dashlinestretch}{30}%
\begin{picture}(4225,1609)(0,-10)%
\put(420,1352){\makebox(0,0)[lb]{\smash{{\SetFigFont{11}{13.2}{\rmdefault}{\mddefault}{\updefault}$x^\ep$}}}}
\put(857,1069){\makebox(0,0)[lb]{\smash{{\SetFigFont{11}{13.2}{\rmdefault}{\mddefault}{\updefault}$d$}}}}
\put(3650,1420){\makebox(0,0)[lb]{\smash{{\SetFigFont{11}{13.2}{\rmdefault}{\mddefault}{\updefault}$x^\ep$}}}}
\put(3494,1074){\makebox(0,0)[lb]{\smash{{\SetFigFont{11}{13.2}{\rmdefault}{\mddefault}{\updefault}$d$}}}}
\put(2159,608){\makebox(0,0)[lb]{\smash{{\SetFigFont{11}{13.2}{\rmdefault}{\mddefault}{\updefault}$\longrightarrow$}}}}
\end{picture}
}
\caption{}\label{fig:coef-shift1}
\end{figure}

\begin{figure}[h]%
\begin{picture}(0,0)%
\includegraphics{fig3.pstex}%
\end{picture}%
\setlength{\unitlength}{0.00087489in}%
\begingroup\makeatletter\ifx\SetFigFont\undefined%
\gdef\SetFigFont#1#2#3#4#5{%
  \reset@font\fontsize{#1}{#2pt}%
  \fontfamily{#3}\fontseries{#4}\fontshape{#5}%
  \selectfont}%
\fi\endgroup%
{\renewcommand{\dashlinestretch}{30}%
\begin{picture}(4453,869)(0,-10)%
\put(929,-40){\makebox(0,0)[lb]{\smash{{\SetFigFont{12}{14.4}{\rmdefault}{\mddefault}{\updefault}$d$}}}}
\put(2042,406){\makebox(0,0)[lb]{\smash{{\SetFigFont{12}{14.4}{\rmdefault}{\mddefault}{\updefault}$\longrightarrow$}}}}
\put(929,700){\makebox(0,0)[lb]{\smash{{\SetFigFont{12}{14.4}{\rmdefault}{\mddefault}{\updefault}$d$}}}}
\put(3154,450){\makebox(0,0)[lb]{\smash{{\SetFigFont{12}{14.4}{\rmdefault}{\mddefault}{\updefault}$x$}}}}
\put(109,450){\makebox(0,0)[lb]{\smash{{\SetFigFont{12}{14.4}{\rmdefault}{\mddefault}{\updefault}$x$}}}}
\end{picture}
}
\caption{}\label{fig:coef-shift2}
\end{figure}

{\em 
Introductions of
new variables:} Nielsen automorphisms of the form $(x^\ep \mapsto x^\ep y^\de)$ where
a variable $y$ does not occur in $Q$. 

{\em Related Nielsen automorphisms:} 
Nielsen automorphisms $(x^\ep \mapsto x^\ep y^\de)$ 
{\em related to $Q$}, that is, those for which
$(x^\ep y^{-\de})^{\pm1}$ occurs in $Q$. Geometrically, 
we shift the start of the edge labelled ~$x^\ep$ along the edge labelled $y^\de$, see Fig.~\ref{fig:var-shift}.
If $(x^\ep y^{-\de})^{\pm1}$ occurs in $Q^{\pm1}$ twice then 
$y$ is eliminated from $Q$. If we view $Q$ as a cyclic word then for any Nielsen automorphism $\rho$
related to $Q$, the image $Q^\rho$ is conjugate to a quadratic word equivalent to $Q$.

\begin{figure}[h]
\begin{picture}(0,0)%
\includegraphics{fig4.pstex}%
\end{picture}%
\setlength{\unitlength}{0.00087489in}%
\begingroup\makeatletter\ifx\SetFigFont\undefined%
\gdef\SetFigFont#1#2#3#4#5{%
  \reset@font\fontsize{#1}{#2pt}%
  \fontfamily{#3}\fontseries{#4}\fontshape{#5}%
  \selectfont}%
\fi\endgroup%
{\renewcommand{\dashlinestretch}{30}%
\begin{picture}(6234,984)(0,-10)
\put(552,102){\makebox(0,0)[lb]{\smash{{\SetFigFont{12}{14.4}{\rmdefault}{\mddefault}{\updefault}$y^\de$}}}}
\put(2082,102){\makebox(0,0)[lb]{\smash{{\SetFigFont{12}{14.4}{\rmdefault}{\mddefault}{\updefault}$y^\de$}}}}
\put(642,597){\makebox(0,0)[lb]{\smash{{\SetFigFont{12}{14.4}{\rmdefault}{\mddefault}{\updefault}$x^\ep$}}}}
\put(2217,597){\makebox(0,0)[lb]{\smash{{\SetFigFont{12}{14.4}{\rmdefault}{\mddefault}{\updefault}$x^\ep$}}}}
\put(1272,237){\makebox(0,0)[lb]{\smash{{\SetFigFont{12}{14.4}{\rmdefault}{\mddefault}{\updefault}$\rightarrow$}}}}
\put(2982,237){\makebox(0,0)[lb]{\smash{{\SetFigFont{12}{14.4}{\rmdefault}{\mddefault}{\updefault}or}}}}
\put(3657,102){\makebox(0,0)[lb]{\smash{{\SetFigFont{12}{14.4}{\rmdefault}{\mddefault}{\updefault}$x^\ep$}}}}
\put(4692,237){\makebox(0,0)[lb]{\smash{{\SetFigFont{12}{14.4}{\rmdefault}{\mddefault}{\updefault}$\rightarrow$}}}}
\put(5457,102){\makebox(0,0)[lb]{\smash{{\SetFigFont{12}{14.4}{\rmdefault}{\mddefault}{\updefault}$x^\ep$}}}}
\put(4152,102){\makebox(0,0)[lb]{\smash{{\SetFigFont{12}{14.4}{\rmdefault}{\mddefault}{\updefault}$y^\de$}}}}
\end{picture}%
}%
\caption{}\label{fig:var-shift}
\end{figure}

It is sometimes convenient to view  transformations of quadratic words from a slightly different 
angle---through formal coefficients. This means that
we introduce a new alphabet ~$C$ of 
{\em formal coefficients}
and consider quadratic words 
in $Q \in F_{C \cup X}$ with the property that each coefficient letter $c\in C$ occurs in $Q$ 
at most once. To get a quadratic word in the usual sense (as an element of $F_{A\cup X})$ we have to
provide it with a {\em coefficient map} $F_C \to F_A$. 

Thus we have two ways of representing quadratic words: as an element of $F_{A\cup X}$ and as a pair
$(Q,\ga)$ where $Q$ is a quadratic word with formal coefficients and $\ga$ is a coefficient map.
An advantage of the second way is that quadratic words can be viewed independently on 
the coefficient group which may be not necessarily free.
We use this representation
since it provides
a more convenient accounting of lengths of coefficients occurring in transformations.%

Note that all elementary transformations introduced above are applicable also to quadratic words with
formal 
coefficients (just with
$F_A$ replaced by the formal coefficient group ~$F_C$). If 
a pair $(Q,\ga)$ represents a quadratic word $R \in F_{A \cup X}$ then transformations of $Q$ 
induce corresponding transformations of $R$. 

In the rest of the section we assume that quadratic 
words are ones with formal coefficients and belong to the group $F_{C \cup X}$. 

We prove a series of statements 
asserting 
that a quadratic word can be reduced to a standard form
using 
an automorphism
of bounded complexity. 
Note that the notion of a standard quadratic word 
is essentially not changed when passing to words with formal coefficients. 
It is required only that any formal coefficient occurs in the coefficients $c_i$ of a standard quadratic word
at most once in total.

\begin{definition}
We call automorphisms of the form $(x^\ep \to W x^\ep)$ where $x$ does not occur in ~$W$, {\em elementary}.

We say that an elementary automorphism $(x^\ep \mapsto W x^\ep)$ {\em changes} $x$ and 
{\em touches} variables occurring in ~$W$.

A product $\rho_1\rho_2\dots\rho_k$ of elementary automorphisms $\rho_i$ is {\em triangular} if it satisfies the
following condition: as soon as $\rho_i$ touches $x$, the variable $x$ is not changed by all subsequent 
automorphisms $\rho_{i+1}$, $\dots$, $\rho_k$. 
\end{definition}

The following observation is immediate:

\begin{lemma} \label{lm:triangular-bound}
Let 
$$
\phi = (x_1^{\ep_1} \to W_1 x_1^{\ep_1}) (x_2^{\ep_2} \to W_2 x_2^{\ep_2}) \dots (x_k^{\ep_k} \to W_k x_k^{\ep_k})
$$ 
be a triangular product of elementary automorphisms of a free group $F$.
For a generator ~$x$, let $W_{i_1}$, $W_{i_2}$, $\dots$, $W_{i_r}$ be all words $W_i$ 
which participate in automorphisms $(x_i \to W_i x_i)$ with $x_i=x$ and $\ep_i = 1$,
and $W_{j_1}$, $W_{j_2}$, $\dots$, $W_{j_t}$ be all words $W_i$ 
which participate in automorphisms $(x_i^{-1} \to W_i x_i^{-1})$ with $x_i=x$ and $\ep_i = -1$.
Then
$$
  x^\phi = W_{i_1} W_{i_2} \dots W_{i_r} x W_{j_t}^{-1} W_{j_{t-1}}^{-1} \dots W_{j_1}^{-1}.
$$
In particular, for any generator $y\ne x$, the number of occurrences of $y$ in $x^\phi$ 
does not exceed the total number of occurrences of $y$ in all $W_{i_k}$ and $W_{j_k}$.
\end{lemma}

\begin{proposition} \label{prop:standard-red-bound-orientable}
Let $Q \in F_{C\cup X}$ be an orientable quadratic word. 
Then there exists an automorphism $\phi \in \Aut_{F_C}(F_{C\cup \Var(Q)})$
such that $Q^\phi$ is conjugate to a standard quadratic word equivalent to ~$Q$ and for any
$x,y \in \Var(Q)$ and $c \in C$, we have $|x^\phi|_y \le 4$ and $|x^\phi|_c \le 2$.

In particular, $|x^\phi| \le 2|Q|$ for any $x\in \Var(Q)$, 
\end{proposition}

\begin{proof}
We construct the required automorphism $\phi$ as
a triangular product of elementary automorphisms.
Starting from this point 
throughout the proof, we denote by ~$Q$ 
the current quadratic word after application of a sequence of elementary automorphisms 
constructed so far. At the start, $Q$ 
is any quadratic word from 
the hypothesis of the proposition.
We view ~$Q$ as a cyclic word and thus 
regard transformations up to conjugation.

We assume that $Q$ has at least one variable (otherwise the proposition is trivial).

At any moment, 
there are {\em locked} variables in $Q$ which have been touched by previous
elementary automorphisms.
They should not be changed by subsequent elementary automorphisms.

By $\Ga_Q$ we denote the 1-skeleton of $S_Q$, i.e.\ the graph embedded in the surface after 
the identification of arcs in the boundary of disk $D$ as described in Section \ref{sec:surfaces}.

{\em Step}\/ 1: Eliminating boundary superfluous vertices.

If $S_Q$ is a closed surface then Step 1 is void and we jump to Step 2. 
Recall that $S_Q$ is closed if and only if $Q$ is coefficient-free.

We call a vertex $\nu$ in the boundary of $S_Q$ {\em essential} 
if it is an endpoint of an interior edge of $S_Q$ 
(i.e.\ an endpoint of an edge 
labelled by a variable). 

Let $\ell$ be a boundary 
component
of $S_Q$. 
Observe that $\ell$ has at least one essential vertex (since $Q$ has at least one variable).
We describe a sequence of elementary automorphisns which result in exactly one essential vertex in $\ell$.
Let $\nu_1$, $\dots$, $\nu_k$ be all cyclically ordered essential vertices in $\ell$.
Let $c\in F_A$ be the label of the arc between $\nu_1$ and $\nu_2$ and $x_1,\dots,x_r \in X^{\pm1}$ be
labels of edges starting at $\nu_2$ so that $c x_1$, $x_1^{-1} x_2$, \dots, $x_{r-1}^{-1} x_r$ 
occur in $Q$
(see Fig.\ \ref{fig:nonessential-vertex}). We apply to ~$Q$ a sequence of 
elementary 
automorphisms
$$
  (x_1\mapsto c^{-1} x_1) (x_2\mapsto c^{-1} x_2) \dots (x_r\mapsto c^{-1} x_r) 
$$
eliminating essential vertex $\nu_2$. Then we proceed in the same way 
eliminating 
all other 
essential
vertices $\nu_3,\dots,\nu_k$.

\begin{figure}[h]%
\begin{picture}(0,0)%
\includegraphics{fig5.pstex}%
\end{picture}%
\setlength{\unitlength}{0.00087489in}%
\begingroup\makeatletter\ifx\SetFigFont\undefined%
\gdef\SetFigFont#1#2#3#4#5{%
  \reset@font\fontsize{#1}{#2pt}%
  \fontfamily{#3}\fontseries{#4}\fontshape{#5}%
  \selectfont}%
\fi\endgroup%
{\renewcommand{\dashlinestretch}{30}%
\begin{picture}(3984,958)(0,-10)%
\put(282,75){\makebox(0,0)[lb]{\smash{{\SetFigFont{12}{14.4}{\rmdefault}{\mddefault}{\updefault}$\nu_1$}}}}
\put(2712,75){\makebox(0,0)[lb]{\smash{{\SetFigFont{12}{14.4}{\rmdefault}{\mddefault}{\updefault}$\nu_1$}}}}
\put(607,75){\makebox(0,0)[lb]{\smash{{\SetFigFont{12}{14.4}{\rmdefault}{\mddefault}{\updefault}$c$}}}}
\put(3037,75){\makebox(0,0)[lb]{\smash{{\SetFigFont{12}{14.4}{\rmdefault}{\mddefault}{\updefault}$c$}}}}
\put(1047,15){\makebox(0,0)[lb]{\smash{{\SetFigFont{12}{14.4}{\rmdefault}{\mddefault}{\updefault}$\nu_2$}}}}
\put(3472,15){\makebox(0,0)[lb]{\smash{{\SetFigFont{12}{14.4}{\rmdefault}{\mddefault}{\updefault}$\nu_2$}}}}
\put(507,515){\makebox(0,0)[lb]{\smash{{\SetFigFont{12}{14.4}{\rmdefault}{\mddefault}{\updefault}$x_1$}}}}
\put(972,525){\makebox(0,0)[lb]{\smash{{\SetFigFont{12}{14.4}{\rmdefault}{\mddefault}{\updefault}$x_r$}}}}
\put(2892,565){\makebox(0,0)[lb]{\smash{{\SetFigFont{12}{14.4}{\rmdefault}{\mddefault}{\updefault}$x_1$}}}}
\put(3387,525){\makebox(0,0)[lb]{\smash{{\SetFigFont{12}{14.4}{\rmdefault}{\mddefault}{\updefault}$x_r$}}}}
\put(722,615){\makebox(0,0)[lb]{\smash{{\SetFigFont{12}{14.4}{\rmdefault}{\mddefault}{\updefault}...}}}}
\put(3146,610){\makebox(0,0)[lb]{\smash{{\SetFigFont{12}{14.4}{\rmdefault}{\mddefault}{\updefault}...}}}}
\put(1767,435){\makebox(0,0)[lb]{\smash{{\SetFigFont{12}{14.4}{\rmdefault}{\mddefault}{\updefault}$\longrightarrow$}}}}
\end{picture}
}
\caption{}\label{fig:nonessential-vertex}
\end{figure}

%
%

We repeat the procedure for all other boundary components of $S_Q$. After that,
each boundary component will have exactly one essential vertex. Their labels are the standard
coefficients of $Q$. 
There are no locked variables after this step.

{\em Step}\/ 2: Eliminating inner superfluous vertices.

We choose a base vertex $\nu_1$ of $\Ga_Q$ as follows. If $Q$ is coefficient-free we take any vertex of ~$\Ga_Q$.
If $Q$ has a coefficient then for the base vertex 
we take any essential vertex in the boundary of ~$S_Q$.

Suppose that $\Ga_Q$ has a vertex $\nu\ne \nu_1$ in the interior of $S_Q$. 
Let $e_1$, $\dots$, $e_k$ be all directed edges starting at $\nu$,
labelled by variables $x_1, \dots, x_k \in X^{\pm1}$.
Since $\nu$ and $\nu_1$ are connected by a path in ~$\Ga_Q$, 
at least one of $e_i$, say $e_1$, ends in a vertex $\nu'$ distinct from $\nu$. 
Then we apply a sequence of elementary automorphisms
$$
  \psi = (x_2\mapsto x_1^{-1} x_2) \dots (x_k\mapsto x_1^{-1} x_k)
$$
eliminating $\nu$.
Observe that $x_1$ does not occur in the new word $Q^\psi$, so the locked variable ~$x_1$ will not participate 
in the subsequent automorphisms.

We perform elimination of all vertices $\nu\ne \nu_1$
in the interior of $S_Q$.
After that, if $Q$ is coefficient-free
then $\Ga_Q$ has only one vertex ~$\nu_1$. 
If $Q$ has a coefficient then $\Ga_Q$ has $m$ vertices 
$\nu_1$, $\dots$, $\nu_{m}$, one in each boundary component of $S_Q$.

{\em Step}\/ 3: Collecting coefficient factors. 

We assume here that $Q$ has at least one coefficient. If $Q$ is coefficient-free, we jump to Step 4.

The step consists of a sequence of substeps $3_1$, $3_2$, $\dots$, $3_{m-1}$.
Before step $3_i$,  $Q$ has the form
$$
  Q = c_1 z_2^{-1} c_2 z_2 z_3^{-1} c_3 z_3 \dots z_{i}^{-1} c_{i} z_{i} W
$$
where $z_j \in X^{\pm1}$ and $c_j \in F_A$ and $W$ has no locked variables. 
Recall that $Q$ is viewed as a cyclic word, so at the start of step 3 we
have $Q = c_1 W$ for some $W$ and coefficient $c_1$.

If no coefficients occur in $W$ then we stop. 
Suppose that a coefficient occurs in $W$. 
Since $\Ga_Q$ is connected, there is an edge in $\Ga_Q$ starting at $\nu_1$ and ending 
in a boundary component of ~$S_Q$ distinct from ones labelled by 
$c_1$, $c_2$, $\dots$, $c_{i}$. Let $z_{i+1} \in X^{\pm1}$ be such an edge 
(for convenience we identify edges with their labels) and
$\nu_{i+1}$ its endpoint in a boundary component labelled ~$c_{i+1}$. 

Step $3_i$ consists of the following.
Using related Nielsen automorphisms and coefficient shifts, 
we first shift starting vertices of all interior edges of $S_Q$ at $\nu_{i+1}$ other 
than $z_{i+1}^{-1}$ to a new position at $\nu_1$ along 
the path labelled $z_{i+1}^{-1}$ or the path labelled $c_{i+1} z_{i+1}^{-1}$ (see Fig.~\ref{fig:step3i}).
Next, if there are edges starting at $\nu_1$ between $z_{i}$ and $z_{i+1}$, we shift
them one by one along the path labelled $z_{i+1} c_{i+1} z_{i+1}^{-1}$.

\begin{figure}[h] 
\begin{picture}(0,0)%
\includegraphics{fig6.pstex}%
\end{picture}%
\setlength{\unitlength}{0.00087489in}%
\begingroup\makeatletter\ifx\SetFigFont\undefined%
\gdef\SetFigFont#1#2#3#4#5{%
  \reset@font\fontsize{#1}{#2pt}%
  \fontfamily{#3}\fontseries{#4}\fontshape{#5}%
  \selectfont}%
\fi\endgroup%
{\renewcommand{\dashlinestretch}{30}%
\begin{picture}(6864,1977)(0,-10)
\put(1883,1100){\makebox(0,0)[lb]{\smash{{\SetFigFont{12}{14.4}{\rmdefault}{\mddefault}{\updefault}$\longrightarrow$}}}}
\put(552,1227){\makebox(0,0)[lb]{\smash{{\SetFigFont{12}{14.4}{\rmdefault}{\mddefault}{\updefault}$\nu_{i+1}$}}}}
\put(10,777){\makebox(0,0)[lb]{\smash{{\SetFigFont{12}{14.4}{\rmdefault}{\mddefault}{\updefault}$z_{i+1}$}}}}
\put(1877,327){\makebox(0,0)[lb]{\smash{{\SetFigFont{12}{14.4}{\rmdefault}{\mddefault}{\updefault}$c_i$}}}}
\put(4377,1100){\makebox(0,0)[lb]{\smash{{\SetFigFont{12}{14.4}{\rmdefault}{\mddefault}{\updefault}$\longrightarrow$}}}}
\put(472,192){\makebox(0,0)[lb]{\smash{{\SetFigFont{12}{14.4}{\rmdefault}{\mddefault}{\updefault}$z_i$}}}}
\put(1047,1822){\makebox(0,0)[lb]{\smash{{\SetFigFont{12}{14.4}{\rmdefault}{\mddefault}{\updefault}$c_{i+1}$}}}}
\end{picture}
}
\caption{}\label{fig:step3i}
\end{figure}

After that, $Q$ gets the form
$$
  Q = c_1 z_2^{-1} c_2 z_2 z_3^{-1} c_3 z_3 \dots z_{i+1}^{-1} c_{i+1} z_{i+1} W'
$$
and we iterate the procedure. Finally we come to a word of the form
$$
  Q = c_1 z_2^{-1} c_2 z_2 \dots z_{m}^{-1} c_{m} z_{m} R
$$
where $R$ is a coefficient-free quadratic word with no locked variables.
If $R$ is empty we get the desired standard quadratic word.
Otherwise we proceed to the next step ~4.

Observe that all edges of $\Ga_Q$ labelled by variables occuring in $R$ start and end at 
the same vertex $\nu_1$.
In this case, we define {\em a star word} $R^*$ as the sequence of labels of edges labelled by variables
in ~$R$ when moving around $\nu_1$ in a small neighborhood of $\nu_1$. To fix the direction
of the motion  we agree that if $x^\ep y^\de$ occurs
in $R^*$ then $y^{-\de} x^\ep$ occurs in ~$R$. If $Q$ has a coefficient then $R^*$ is the word read off
between the edges of $\Ga_Q$ labelled the starting letter of $c_1$ and ~$z_{m}$ 
(see Fig.\ \ref{fig:star-word}).
If $Q$ is coefficient-free then $R=Q$ and we view $R^*$ as a cyclic word.

\begin{figure}[h] 
\begin{picture}(0,0)%
\includegraphics{fig7.pstex}%
\end{picture}%
\setlength{\unitlength}{0.00087489in}%
\begingroup\makeatletter\ifx\SetFigFont\undefined%
\gdef\SetFigFont#1#2#3#4#5{%
  \reset@font\fontsize{#1}{#2pt}%
  \fontfamily{#3}\fontseries{#4}\fontshape{#5}%
  \selectfont}%
\fi\endgroup%
{\renewcommand{\dashlinestretch}{30}%
\begin{picture}(2799,1790)(0,-10)
\put(0,970){\makebox(0,0)[lb]{\smash{{\SetFigFont{12}{14.4}{\rmdefault}{\mddefault}{\updefault}$c_2$}}}}
\put(500,1650){\makebox(0,0)[lb]{\smash{{\SetFigFont{12}{14.4}{\rmdefault}{\mddefault}{\updefault}$c_m$}}}}
\put(820,435){\makebox(0,0)[lb]{\smash{{\SetFigFont{12}{14.4}{\rmdefault}{\mddefault}{\updefault}$z_1$}}}}
\put(1000,805){\makebox(0,0)[lb]{\smash{{\SetFigFont{12}{14.4}{\rmdefault}{\mddefault}{\updefault}$z_m$}}}}
\put(1395,1090){\makebox(0,0)[lb]{\smash{{\SetFigFont{12}{14.4}{\rmdefault}{\mddefault}{\updefault}$x_r^{\ep_r}$}}}}
\put(2150,630){\makebox(0,0)[lb]{\smash{{\SetFigFont{12}{14.4}{\rmdefault}{\mddefault}{\updefault}$x_1^{\ep_1}$}}}}
\put(2445,175){\makebox(0,0)[lb]{\smash{{\SetFigFont{12}{14.4}{\rmdefault}{\mddefault}{\updefault}$c_1$}}}}
\put(1850,990){\makebox(0,0)[lb]{\smash{{\SetFigFont{12}{14.4}{\rmdefault}{\mddefault}{\updefault}$x_2^{\ep_2}$}}}}
\put(1685,1090){\makebox(0,0)[lb]{\smash{{\SetFigFont{12}{14.4}{\rmdefault}{\mddefault}{\updefault}...}}}}
\put(935,690){\makebox(0,0)[lb]{\smash{{\SetFigFont{12}{14.4}{\rmdefault}{\mddefault}{\updefault}...}}}}
\end{picture}
}
$$
  R^* = x_1^{\ep_1} x_2^{\ep_2} \dots x_r^{\ep_r}
$$
\caption{}\label{fig:star-word}
\end{figure}

Observe that $R^*$ is an orientable quadratic word of the same length and with the 
same variables as $R$.

{\em Step}\/ 4: Collecting commutators. 
We assume that $Q$ has a coefficient; the coefficient-free case is similar with obvious minor changes.

At any stage of this step, we do not change the coefficient part of $Q$ already collected, 
so $Q$ has the above form
$$
   Q = c_1 z_2^{-1} c_2 z_2 \dots z_{m}^{-1} c_{m} z_{m} R
$$
where $R$ is a coefficient-free quadratic word. We will apply only related Nielsen automorphisms 
involving variables of $R$.
This implies that all edges of $\Ga_Q$ labelled by variables in $R$ 
start and end at the same vertex $\nu_1$.

The whole step 4 is again an iterative sequence of smaller steps $4_1$, $\dots$, $4_g$. 
Before substep ~$4_i$, we have
$$
  R =  [x_1,y_1] \dots [x_{i-1},y_{i-1}]T
$$
and 
$$
  R^* = T^* [y_{i-1}, x_{i-1}^{-1}] \dots [y_1,x_1^{-1}] 
$$
where $T$ and $T^*$ are orientable quadratic words with $\Var(T) = \Var(T^*)$; in particular, $|T|=|T^*|$.

We stop if $T$ and $T^*$ are empty. 

Assume that $T$ and $T^*$ are nonempty. Let $x_i$ be a variable occurring in ~$T$, so 
$$
  T^* = U x_i V x_i^{-1} W
$$ 
up to the exponent sign of $x_i$.

We claim that at least one variable $y_i$ occurs in $V$ exactly once. Indeed, if $V$ is
a (possibly empty) quadratic word then a quadratic word ~$\bar{V}$ lies between the 
two occurrences of ~$x_i$ in ~$T$, that is, $T = \dots x_i^{\pm1} \bar{V} x_i^{\mp1} \dots$. 
In this case the edge of $\Ga_Q$ labelled $x_i$ would wave distinct endpoints, a contradiction. 

Hence, up to interchanging $x_i$ and $y_i$ and changing their exponent signs we have
$$
  T^* = Z_1 y_i^{-1} Z_2 x_i Z_3 y_i Z_4 x_i^{-1} Z_5
$$

The following sequence of automorphisms collects the commutator $[y_i, x_i^{-1}]$
when applied to $T^*$:
\begin{align*}
T^* \xrightarrow{(x_i \mapsto Z_5 x_i)}\quad 
  &Z_1 y_i^{-1} Z_2 Z_5 x_i Z_3 y_i Z_4 x_i^{-1} \\
\xrightarrow{(y_i \mapsto Z_2 Z_5 u_i)} \quad 
  &Z_1 y_i^{-1} x_i Z_3 Z_2 Z_5 y_i Z_4 x_i^{-1} \\
\xrightarrow{(x_i \mapsto x_i (Z_3 Z_3 Z_5)^{-1})}\quad 
  &Z_1 y_i^{-1} x_i y_i Z_4 Z_3 Z_2 Z_5 x_i^{-1} \\
\xrightarrow {(y_i \mapsto y_i (Z_4 Z_3 Z_3 Z_5)^{-1})}\quad 
  &Z_1 Z_4 Z_3 Z_2 Z_5 y_i^{-1} x_i y_i x_i^{-1} 
\end{align*}

It is straightforward to check that for any Nielsen automorphism $\rho$ related to $T^*$ there 
is a dual Nielsen automorphism $\rho^*$
related to $T$ whose action on $T$ agrees with the action of ~$\rho$ on $T^*$, 
that is, $(T^*)^\rho = (T^{\rho^*})^*$:
If $x^\ep y^\de$ occurs in $T^*$ then we define
$$
  (x^\ep \mapsto x^\ep y ^{-\de})^* = (y^\de \mapsto x^\ep y^\de) \quad\text{and}\quad
  (y^\de \mapsto x^{-\ep} u^\de)^*=(x^\ep \mapsto y^\de x^\ep).
$$
We observe that if $\rho$ changes $x$ and touches $y$, then the role of these variables in $\rho^*$ is
interchanged. 
This implies that there is a sequence $\psi$ of related Nielsen
automorphisms touching only $x_i$ and $y_i$ which transforms~$T$ to a word $T^\psi$ where 
$$
  (T^\psi)^* =  Z [y_i, x_i^{-1}].
$$
Then for some $T'$,
$$
  T^\psi = [x_i,y_i] T'
$$
as required. This finishes step $4_i$.

After completion of this procedure $Q$ gets the standard form. We get an automorphism $\phi$ 
reducing to such a form an arbitrary quadratic word from the hypothesis of the proposition.

{\em Calculation of the bounds for $\phi$}.
Fix a variable $x \in \Var(Q)$. 
Let $W_1$, $W_2$, $\dots$ $W_k$ be the list of all words $W$ in elementary
automorphisms $(x^\ep \to W x^\ep)$ which are the factors of ~$\phi$. 
We claim that any variable ~$y \ne x$ occurs in all $W_i$ in total at most 4 times.
Indeed, at each individual step 2, 3 or 4, $y$ occurs at most twice in the automorphisms
of the form $(x \to W x)$ and at most twice in the automorphisms of the form $(x^{-1} \to W x^{-1})$. 
It remains to 
observe that if $y$ participates (i.e.\ occurs in some $W$) at step 2 or 3 then $y$ does not participate
in subsequent steps. By Lemma \ref{lm:triangular-bound} we get
$$
  |x^\phi|_y \le 4.
$$
From the construction it is easy see also that each constant $c \in C$ occurs in 
$|x^\phi|$ at most twice. This implies 
$$
  |x^\phi|_c \le 2.
$$
This finishes the proof of Proposition \ref{prop:standard-red-bound-orientable}.
\end{proof}

\begin{definition} \label{def:semi-standard}
We call a non-orientable quadratic word $Q$ {\em semi-standard} if $Q$ has one of the following forms
$$
  Q = x_1^2 x_2^2 \dots x_k^2 [x_{k+1}, y_{k+1}] \dots [x_{n}, y_{n}]
$$
or 
$$
  Q = x_1^2 x_2^2 \dots x_k^2 [x_{k+1}, y_{k+1}] \dots [x_{n}, y_{n}] 
  c_1 z_2^{-1} c_2 z_2 \dots z_m^{-1} c_m z_m
$$
where factors $[x_i,y_i]$ are not obligatory.
\end{definition}

\begin{proposition} \label{prop:semistandard-red-bound-nonorientable}
Let $Q \in F_{C\cup X}$ be a non-orientable quadratic word. 
Then there exists an automorphism $\phi \in \Aut_{F_C}(F_{C\cup \Var(Q)})$
such that $Q^\phi$ is conjugate to a semi-standard quadratic word equivalent to ~$Q$ and for any 
$x,y \in \Var(Q)$ and $c \in C$, we have $|x^\phi|_y \le 4$ and $|x^\phi|_c \le 2$.

In particular, $|x^\phi| \le 2|Q|$ for any $x \in \Var(Q)$.
\end{proposition}

\begin{proof}
The steps 1--3 are the same as in the case of orientable $S_Q$. After performing these three steps we get
$$
   Q = c_1 z_2^{-1} c_2 z_2 \dots z_{m}^{-1} c_{m} z_{m} R
$$
where $R$ is now a non-orientable coefficient-free quadratic word. The rest of reduction consists of 
the following step.

{\em Step}\/ $4^{\rm n}$: Collecting squares and commutators.

As in the orientable case, we will work with the star word $R^*$. Unfortunately, the definition of $R^*$ given
in the proof of Proposition \ref{prop:standard-red-bound-orientable} suits only for orientable $R$ since it 
always produces an orientable $R^*$. 
We modify the definition.

For any edge $e$ of $S_Q$, we fix its {\em orientation} which is a choice of the positive direction 
of crossing ~$e$ inside ~$S_Q$.
When passing around the base vertex $\nu_1$, we read the label $x$ of $e$ with exponent $+1$ if we cross $e$ 
in the positive direction
and $-1$ otherwise (so the exponent signs of variables in $R^*$ are not related directly to the 
exponent signs of 
the corresponding occurrences in $R$). It is easy to check that application of a related Nielsen automorphism
$\rho = (x^\ep \mapsto y^\de x^\ep)$ to ~$R^*$ agrees with application of a dual Nielsen 
automorphism ~$\rho^*$ related to ~$R$ that changes $y$ and touches $x$. 

Let $e$ be an edge of $\Ga_Q$ labelled by a variable $x \in \Var(R)$. If $e$ reverses the orientation 
when viewed as a loop in $S_Q$ then we cross $e$ twice in the same direction when passing around ~$\nu_1$.
Hence the both occurrences of $x$ in $R^*$ have the same exponent sign. 
Since there is at least one orientation-reversing $e$, $R^*$ is a non-orientable quadratic word.
Thus, up to the exponent sign of $x$, $R^*$ has the form
$$
  R^* = Z_1 x Z_2 x Z_3.
$$
We then collect the square of $x$:
$$
R^* \xrightarrow{(x \mapsto x Z_3^{-1})}
  Z_1 x Z_3^{-1} Z_2 x 
\xrightarrow{(x \mapsto Z_2^{-1} Z_3 x)}
   Z_1 Z_2^{-1} Z_3 x^2 
$$
If $Z_1 Z_2^{-1} Z_3$ is non-orientable then we repeat the procedure collecting the square
of a variable to the right. Otherwise we apply the procedure of collecting commutators described in 
Step ~4 in the orientable case. Using a triangular sequence of elementary automorphisms we finally reduce ~$Q$ 
to the form
$$
  Q = c_0 z_1^{-1} c_1 z_1 \dots z_{m-1}^{-1} c_{m-1} z_{m-1} x_1^2 \dots x_k^2 
    [x_{k+1},y_{k+1}] \dots [x_r,y_r]
$$
The bound on $\phi$ is obtained in the same way as in the orientable case.
\end{proof}

To reduce a semi-standard non-orientable quadratic word to a standard form we need an extra transformation.

\begin{lemma} \label{lm:extra-autos}
There are automorphisms $\eta_k, \th_k \in \Aut (F_{\set{x_0,\dots,x_k,y_1,\dots,y_k}})$ such that
$$
  x_0^2 [x_1,y_1] \dots [x_k,y_k] \xrightarrow{\eta_k, \th_k} x_0^2 x_1^2 y_1^2 \dots x_k^2 y_k^2
$$
and for any $x \in \set{x_0,\dots,x_k,y_1,\dots,y_k}$, 
$$
  |x^{\eta_k}| \le 4k+1  \quad\text{and}\quad |x^{\th_k^{-1}}| \le 3 k+2.
$$
\end{lemma}

\begin{proof}
Let $\ga(x,y,z)$ be an automorphism of $F_{\set{x,y,z}}$ such that 
$$
  (x^2 [y,z])^\ga = x^2 y^2 z^2.
$$
Then taking
$$
  \eta_k = \ga(x_0,x_1,y_1) \ga(y_1,x_2,y_2) \dots \ga(y_{k-1},x_k,y_k)
$$
we obviously get 
$$
  x_0^2 [x_1,y_1] \dots [x_k,y_k] \xrightarrow{\eta_k} x_0^2 x_1^2 y_1^2 \dots x_k^2 y_k^2
$$
For a specific $\ga$, we take
$$
  \ga = (x \mapsto x^2 y z x^{-1}, \ y \mapsto xyz x^{-1}, \ z \mapsto xz).
$$
The bound $||\eta_k|| \le 4k+1$ is straightforward. 

To define $\th_k$ we proceed in a similar way by taking for $\ga$ another automorphism 
$$
  \ga = (x \mapsto xyz, \ 
	y \mapsto z^{-1} y^{-1} x^{-1} y z xyz, \
	z \mapsto z^{-1} y^{-1} x^{-1} z)
$$
with
$$
  \ga^{-1} = (x \mapsto x^2 y^{-1} x^{-1}, \ y \mapsto xyx^{-1} z^{-1} x^{-1}, \ z \mapsto xz).
$$

\end{proof}

\begin{proposition} \label{prop:standard-red-bound-nonorientable}
Let $Q \in F_{C\cup X}$ be a non-orientable quadratic word. 
Then there exists an automorphism $\phi \in \Aut_{F_C}(F_{C\cup \Var(Q)})$
such that $Q^\phi$ is conjugate to a standard quadratic word equivalent to ~$Q$ and 
for any $x \in \Var(Q)$ we have $|x^\phi|_X \le 8 n(Q) \,\mathrm{genus}(Q)$ and 
$|x^\phi|_c \le 2$ for any formal coefficient $c \in C$.
\end{proposition}

\begin{proof}
Proposition \ref{prop:semistandard-red-bound-nonorientable} and Lemma \ref{lm:extra-autos} produce
an automorphism $\phi \in \Aut_{F_C}(F_{C\cup \Var(Q)})$ such that 
$Q^\phi$ is conjugate to a standard quadratic word equivalent to ~$Q$ and 
$$
  |x^\phi|_X \le 4 n(Q) (4k+1) \quad\text{and}\quad  |x^\phi|_c \le 2
$$
where $k$ is the number of commutators in the semi-standard form given by 
Proposition \ref{prop:semistandard-red-bound-nonorientable}.
It remains to notice that
$$
  k \le \frac12 (\text{genus} (Q) - 1).
$$
\end{proof}

We turn now to bounds similar to Propositions \ref{prop:standard-red-bound-orientable}, 
\ref{prop:semistandard-red-bound-nonorientable} and \ref{prop:standard-red-bound-nonorientable}
where we estimate the 
size of the automorphism {\em inverse} to $\phi$. Note that
the sizes of an automorphism $\phi$ of a free group ~$F$ and of its inverse $\phi^{-1}$ can be 
very different. We give an example where the ratio is exponential in the rank of $F$.

\begin{example} \label{exmpl:inverse-streching}
Let $F = F_{\set{x_1, \dots, x_n}}$. Define an automorphism $\phi\in\Aut(F)$ by
\begin{align*}
  x_{2i+1}^\phi &= x_{i+1} x_i x_{i+2} x_{i-1} \dots x_{2i} x_1 x_{2i+1}, \\
  x_{2i}^\phi &= x_i x_{i+1} x_{i-1} x_{i+2} \dots x_1 x_{2i}
\end{align*}
Then $||\phi|| = n$ 
where, by definition, $||\phi|| = \max_i |x_i^\phi|$.
It is not hard to see that $||\phi^{-1}|| = 2^n$.
\end{example}

\begin{proposition} \label{prop:semistandard-red-bound-inverse}
Let $Q \in F_{C \cup X}$ be any quadratic word. 
Then there exists an automorphism $\psi \in \Aut_{F_C}(F_{C\cup \Var(Q)})$
such that $Q^\psi$ is conjugate to a quadratic word $R$ equivalent to $Q$ and 
the following assertions are true:
\begin{enumerate}
\item 
$R$ is standard if $Q$ is orientable and semi-standard if $Q$ is non-orientable.
\item
For any $x,y \in \Var(Q)$ and $c \in C$, we have
$|x^{\psi^{-1}}|_y \le 4$ and $|x^{\psi^{-1}}|_c \le 2$.
In particular, $|x^{\psi^{-1}}| \le 2|Q|$.
\end{enumerate}
\end{proposition}

\begin{proof}
Similarly to the arguments used in the proof of Propositions \ref{prop:standard-red-bound-orientable}
and \ref{prop:semistandard-red-bound-nonorientable},
we view $Q$ as a cyclic word and construct $\psi$ 
so that the inverse automorphism $\psi^{-1}$ will be a triangular product
of elementary automorphisms.
The condition that $\psi^{-1}$ is triangular is equivalent to the condition that $\psi$
is {\em reverse triangular} in the following sense: if a variable $x$ is changed by some 
elementary automorphism in
the product then $x$ is not touched by subsequent elementary automorphisms. 
If fact, the proof will be simpler than the proof of 
Propositions ~\ref{prop:standard-red-bound-orientable} and ~\ref{prop:semistandard-red-bound-nonorientable}
because there is no need to pass to the star word and we will operate on the quadratic word ~$Q$ itself. 

According to the change in the notion of a triangular product, we change the notion of a locked variable: 
a variable $x$ is viewed as locked at a current transformation step if it was previously 
changed by an elementary automorphism.
Subsequent elementary automorphisms should not touch locked variables. 
Similar to the proof of Propositions ~\ref{prop:standard-red-bound-orientable} and
\ref{prop:semistandard-red-bound-nonorientable}, 
we collect locked variables to a fixed part of $Q$. Thus, at any moment
$Q$ has the form $LR$ where $L$ is the locked part of $Q$ and $R$ has no locked variables. 
At the start, $L$ is empty
and since $Q$ is viewed a cyclic word,  
we assume without loss of generality that $R$ starts with a coefficient letter.

We proceed in two steps.

{\em Step\/} $1^{\rm i}$: Reducing the coefficient-free part.

If $R$ is non-orientable then $R = Z_1 x Z_2 x Z_3$ for some variable $x$. Using a sequence 
of elementary automorphisms similar
to one given in Step $4^{\rm n}$ in the proof of Proposition \ref{prop:semistandard-red-bound-nonorientable}
we collect $x^2$ to the left of $R$ and add it to the locked part.

If $R$ is orientable and there are variables $x$ and $y$ which ``cross'' in $R$, that is,
$R = Z_1 x^{-1} Z_2 y^{-1} Z_3 x Z_4 y Z_5$ (up to exponent signs of $x$ and $y$) then we collect the
commutator $[x,y]$ to the left of $R$ as in Step 4 in the proof of 
Proposition \ref{prop:standard-red-bound-orientable}.

Iterating the procedure, we reduce $Q$ to the form
$$
  Q = x_1^2 \dots x_k^2 [x_{k+1},y_{k+1}] \dots [x_r,y_r] R
$$
where $R$ is orientable and has no crossing variables.

{\em Step\/} $2^{\rm i}$: Collecting the coefficient part. 

If $R$ has no variables then we are done. Let $\Var(R)\ne\eset$.
By the assumption that $R$ has no crossing variables, there is a variable $z$ with no variables
between the two occurrences of $z$ in ~$R$. Then we have $R = T_1 z^{-1} c z T_2$ where $c \in F_A$ is a standard
coefficient of $Q$. Applying automorphism 
$(z \mapsto z T_2^{-1})$ to $Q$ we shift $z^{-1} c z$ to the right of $R$. Iterating the
procedure we get
$$
  Q = x_1^2 \dots x_k^2 [x_{k+1},y_{k+1}] \dots [x_r,y_r] c_1 z_2^{-1} c_2 z_2 \dots z_m^{-1} c_m z_m
$$
It remains to observe that at any transformation step, 
we preserve the property that $R$ starts with a coefficient letter and so $c_1\ne 1$.

The required bounds $|x^{\psi^{-1}}|_y \le 4$ and $|x^{\psi^{-1}}|_c \le 2$ are straighforward 
in view of Lemma \ref{lm:triangular-bound} (applied to $\psi^{-1}$).
\end{proof}

From Proposition \ref{prop:semistandard-red-bound-inverse} and Lemma \ref{lm:semistandard-transition}
we get a dual version of Proposition ~\ref{prop:standard-red-bound-nonorientable}.

\begin{proposition} \label{prop:standard-red-bound-nonorientable-inverse}
Let $Q \in F_{C \cup X}$ be a non-orientable quadratic word. 
Then there exists an automorphism $\psi \in \Aut_{F_C}(F_{C\cup \Var(Q)})$
such that $Q^\psi$ is conjugate to a standard quadratic word equivalent to $Q$ and 
for any $x,y \in \Var(Q)$ and $c \in C$, we have $|x^{\psi^{-1}}|_y \le 8 \,\mathrm{genus}(Q)$ and 
$|x^{\psi^{-1}}|_c \le 4 \,\mathrm{genus}(Q)$. 
\end{proposition}

Passing from quadratic words with formal coefficients to quadratic words $Q \in G * F_X$ 
over a group $G$ we can easily formulate a general result for an {\em arbitrary} coefficient group $G$. 

\begin{corollary} \label{cor:standard-reduction-general} 
Let $Q \in G * F_X$ be quadratic word over an arbitrary group $G$. Then
there are automorphisms $\phi,\psi \in \Aut_G(G * F_X)$
such that $Q^\phi$ and $Q^\psi$ are conjugate to standard quadratic words equivalent to $Q$ 
and for any variable $x$,
$$
  |x^\phi| , |x^{\psi^{-1}}| \le
    \begin{cases}
    4 n(Q) + 2 c(Q) & \quad \text{if $Q$ is orientable} \\
    8 n^2 (Q) + 4 n(Q) c(Q)  & \quad \text{if $Q$ is non-orientable}
    \end{cases}
$$
where $c(Q)$ is the total length of coefficients of $Q$ expressed in any left-invariant 
(e.g.\ word) metric on $G$.
\end{corollary}

\begin{proof}
For orientable $Q$ this follows from Propositions \ref{prop:standard-red-bound-orientable} and
\ref{prop:semistandard-red-bound-inverse}. If $Q$ is non-orientable then we have to use
Propositions \ref{prop:standard-red-bound-nonorientable} and \ref{prop:standard-red-bound-nonorientable-inverse}
and inequality $\text{genus}(Q) \le n(Q)$.
\end{proof}

In the end of the section, we formulate several lemmas for later use. The first one was in fact proved 
in the proof of Proposition \ref{prop:semistandard-red-bound-inverse} 
(unlike the proposition, we do not make a passage to a conjugate element here).

\begin{lemma} \label{lm:cffree-reduction}
\begin{enumerate}
\item 
Let $Q \in F_X$ be a coefficient-free orientable quadratic word. 
Then there is an automorphism $\psi\in \Aut(F_X)$ such that 
$$
  Q^\psi = [x_1, y_1] \dots [x_g, y_g]
$$
and $|x_i^{\psi^{-1}}|_x, |y_i^{\psi^{-1}}|_x \le 4$ for all $i$ and $x \in \Var(Q)$. 
\item
Let $Q \in F_X$ be a coefficient-free non-orientable quadratic word. 
Then there is an automorphism $\psi\in \Aut(F_X)$ such that 
$$
  Q^\psi = x_1^2 \dots x_k^2 [x_{k+1}, y_{k+1}] \dots [x_{k+n}, y_{k+n}]
$$
and $|x_i^{\psi^{-1}}|_x, |y_i^{\psi^{-1}}|_x \le 4$ for all $i$ and $x \in \Var(Q)$. 
\end{enumerate}
\end{lemma}

\begin{lemma} \label{lm:permut-coefficients}
Let $m\ge 1$ and $\si$ be a permutation on $\set{1,\dots,m}$. Then there is an automorphism $\psi$ of
$F_{\set{c_1,\dots,c_{m},z_1,\dots,z_{m}}}$ such that $c_i^\psi = c_i$ for all $i$, with the 
following properties:
\begin{enumerate}
\item 
$$
  z_1^{-1} c_1 z_1 z_2^{-1} c_2 z_2 \dots z_{m}^{-1} c_{m} z_{m} \xrightarrow{\psi}
    z_{\si(1)}^{-1} c_{\si(1)} z_{\si(1)} z_{\si(2)}^{-1} c_{\si(2)} z_{\si(2)} \dots 
      z_{\si(m)}^{-1} c_{\si(m)} z_{\si(m)}.
$$
\item
For any $i$, the image $z_i^\psi$ of $z_i$ has the form
$$
  z_i^\psi = z_i z_{\si(j_1)}^{-1} c_{\si(j_1)} z_{\si(j_1)} \dots 
    z_{\si(j_k)}^{-1} c_{\si(j_k)} z_{\si(j_k)}
$$
for some increasing sequence of indices $1 \le j_1 <\dots< j_k \le m$.
\end{enumerate}
\end{lemma}

\begin{proof}
Using the automorphism
$$
  z_i^{-1} c_i z_i z_j^{-1} c_j z_j \xrightarrow{(z_j \mapsto z_j z_i^{-1} c_i z_i)} 
  z_j^{-1} c_j z_j z_i^{-1} c_i z_i
$$
we can permute two neighboring factors of the form $z_i^{-1} c_i z_i$.
Starting with $z_{\si(m-)}^{-1} c_{\si(m)} z_{\si(m)}$ we arrange factors $z_i^{-1} c_i z_i$
to the right in the order as they should occur in the desired image of 
$z_1^{-1} c_1 z_1 \dots z_{m}^{-1} c_{m} z_{m}$. 
We get a triangular product of elementary automorphisms and
all ~$z_i^\psi$ have the required form as stated in (ii).
\end{proof}

\begin{lemma} \label{lm:semistandard-transition}
Let 
$$
  Q_1 = x_1^2 \dots x_k^2 [x_{k+1}, x_{k+2}] \dots [x_{g-1}, x_{g}] 
    z_1^{-1} c_1 z_1 z_2^{-1} c_2 z_2 \dots z_{m}^{-1} c_{m} z_{m}
$$
and 
$$
  Q_2 = x_1^2 \dots x_r^2 [x_{r+1}, x_{r+2}] \dots [x_{g-1}, x_{g}] 
    z_1^{-1} c_1^{\ep_0} z_1 z_2^{-1} c_2^{\ep_1} z_2 \dots z_{m}^{-1} c_{m}^{\ep_{m}} z_{m}, \quad
  \ep_1,\dots,\ep_{m} = \pm1,
$$
be two non-orientable quadratic words of the same genus $g$ with the same 
set of coefficient letters $\set{c_1,\dots,c_m}$.
Then there is an automorphism $\psi$ of
$F_{\set{c_1,\dots,c_{m}} \cup \Var(Q_1)}$ such that $c_i^\psi = c_i$ for all ~$i$, with the 
following properties:
\begin{enumerate}
\item 
$Q_1^\psi$ is conjugate to $Q_2$.
\item 
For any variable $y \in \set{x_1,\dots,x_g,z_1,\dots,z_{m}}$,
$$
  |y^\psi|_Y \le 4g, \quad |y^\psi|_Z \le 4m \quad\text{and}\quad |y^\psi|_{c_i} \le 2, \ i=1,\dots,m
$$
where $Y = \set{x_1,\dots,x_g}$ and $Z=\set{z_1,\dots,z_{m}}$.
\end{enumerate}
\end{lemma}

\begin{proof}
We use the following set of automorphisms: 
\begin{align*}
  x^2 [y,z] &\xrightarrow{\phi_1} y^2 z^2 x^2, &
    \phi_1 &= (x \mapsto y^2zxy^{-1}, \ y \mapsto y z x y^{-1}, \ z \mapsto y x),\\
  x^2 [y,z] &\xrightarrow{\phi_2} [y,z] x^2, &
    \phi_2 &= (x \mapsto [y,z] x [y,z]^{-1}),\\
  x^2 y^2 z^2 &\xrightarrow{\phi_3} [y,z] x^2, &
    \phi_3 &= (x \mapsto [y,z] x y, \ y \mapsto y^{-1} x^{-1} z^{-1}, \ z \mapsto zx),\\
  x^2 z^{-1} c z &\xrightarrow{\phi_4} z^{-1} c^{-1} z x^2 , &
    \phi_4 &= (x \mapsto z^{-1} c^{-1} z x, \ z \mapsto zx),\\
  x^2 z^{-1} c z &\xrightarrow{\phi_5} z^{-1} c z x^2, &
    \phi_5 &= (z \mapsto z x^2).
\end{align*}
We pick up $x_q^2$ with $q = \min\set{k,r}$ and move it to the right of $Q_1$ so that the square/commutator
part becomes the same as in $Q_2$ and each constant $c_i$ gets the exponent sign ${\ep_i}$ as in $Q_2$.
The bounds in (ii) are straightforward.
\end{proof}

\section{An elimination process for quadratic words} \label{sec:simple-elimination}

There is a general approach to  solutions of equations in free and similar groups which may be called
{\em elimination process}. It usually deals with pairs $(E,\al)$ where $E=1$ is an equation in a 
group $F$ and $\al$ is its solution. On the set of such pairs, a certain set of transformations is defined.
The idea is to reduce step-by-step the cancellation which appears
after substituting ~$\al$ into $E$ and then, starting from a given pair 
$(E,\al)$, to get a new pair $(E_1,\al_1)$ of bounded complexity. In particular, solvability
of equation $E=1$ is reduced to existence of a solution of bounded complexity 
of finitely many such equations $E_1=1$ which can be algorithmically checked.
To describe the solution set of $E=1$ one needs in addition to track transformations homomorphisms. 
(There are several ways to define them; for example, in the case of a free group $F=F_A$
a transformation homomorphism from a pair $(E,\al)$ to a pair $(E_1,\al_1)$ 
can be viewed as an endomorphism $\phi: F_{A\cup X} \to F_{A\cup X}$ 
such that $E^\phi = E_1$ and $\al = \phi\al_1$.)
For free groups, an elimination process has been applied in 
its full strength by Razborov \cite{Razborov-1987} to provide
a description of solutions of an arbitrary system of equations in free groups. In the case of quadratic equations,
there is a much simpler version in \cite{ComerfordEdmunds-1981,ComerfordEdmunds-1989,Grigorchuk-Kurchanov-1989}.

In this section, we apply a version of the elimination process to pairs of the form $(Q,\al)$ where 
$Q \in F_{X}$ is a coefficient-free quadratic word and $\al$ is an evaluation of variables of $Q$, 
i.e.\ a homomorphism $F_X \to F_A$. We use the process to give
a bound on the size of {\em some} solution of a related quadratic equation and thus do not need
to track transformation homomorphisms. A more sophisticated version 
(with tracking transformation homomorphisms) 
will be used in Section~\ref{sec:parametric} to obtain bounds on the size of parametric solutions.

Let $(Q,\al)$ be a pair where $Q \in F_X$ is a coefficient-free quadratic word 
and $\al: F_X \to F_A$ is an evaluation of variables of $Q$.
We define several types of elementary transformations which carry $(Q,\al)$
to a new pair $(Q_1,\al_1)$. In all cases, there will be a transformation homomorphism $\phi \in \End(F_X)$
such that $Q_1 = Q^\phi$ (but we do not require that $\al = \al_1 \phi$). As one of the main inductive
parameters, we take the length of a formal word $Q[\al]$ obtained by substituting in $Q$ values $x^\al$
of all variables $x \in \Var(Q)$ represented by freely reduced words. It is straightforward to check that 
$|Q_1[\al_1]| \le |Q[\al]|$ for all elementary transformations introduced below.

{\em Degenerate transformation}. Application condition: $x^\al = 1$ for some $x \in \Var(Q)$.
We do not change $\al$ and apply the homomorphism $\phi = (x\mapsto 1)$ to $Q$, so $Q_1$ is 
obtained from $Q$ by removing
all occurrences of $x$ and performing subsequent cancellation.

We introduce a natural notion related to this transformation. We say that a coefficient-free quadratic
word $Q_2$ is a {\em homomorphic image} of $Q_1$ if there is an endomorphism 
$\phi\in\End(F_X)$ with $Q_1^\phi=Q_2$. The following fact is an easy exercise.

\begin{proposition}
Let $Q_1$ and $Q_2$ be coefficient-free quadratic words. Then $Q_2$ is a homomorphic image of $Q_1$
if and only if either $Q_1$ and $Q_2$ are of the same orientability
and $genus(Q_2)\le genus(Q_1)$ or $Q_1$ is non-orientable, $Q_2$ is orientable and 
$genus(Q_2) < \frac12 genus(Q_1)$.
\end{proposition}

{\em Cancellation reduction}. 
Application condition: a non-trivial cancellation in $Q[\al]$ occurs between 
the values of two neighboring variables $x^\ep$ and $y^\de$ ($\ep,\de=\pm1$). There are two cases.

{\em Case\/} 1: $x \ne y$. Then for some $u, v$ and $w\ne 1$,
$$
  (x^\ep)^\al = uw \quad\text{and}\quad (y^\de)^\al = w^{-1} v
$$
where equality stands for graphical equality of words. 
We take a variable $z \notin \Var(Q)$ and define a transformation homomorphism $\phi$ by
$$
  \phi = (x^\ep \mapsto x^\ep z, \ y^\de \mapsto z^{-1} y^\de).
$$
To define $\al_1$ we set
$$
  (x^\ep)^{\al_1} = u, \quad (y^\de)^{\al_1} = v, \quad  z^{\al_1} = w
$$
and $h^{\al_1} = h^\al$ for all other variables $h\ne x,y,z$. 

{\em Case\/} 2: $x=y$. Then $x^\ep y^\de$ becomes $(x^\ep)^2$ and for some $u$ and $w \ne 1$,
$$
  (x^\ep)^\al = w^{-1} u w.
$$
In a similar manner, we take
$$
  \phi = (x^\ep \mapsto z^{-1} x^\ep z)
$$
and define $\al_1$ by 
$$
(x^\ep)^{\al_1} = u, \quad z^{\al_1} = w.
$$
After application of this transformation we get a pair $(Q_1,\al_1)$ which satisfies the strict
inequality $|Q_1[\al_1]| < |Q[\al]|$.

{\em Splitting a variable.} Application condition: $x$ is a variable in $Q$ with $x^\al \ne 1$.
We introduce a new variable $y \notin \Var(Q)$ and apply to $Q$ the 
substitution $\phi = (x \mapsto xy)$.
For $\al_1$, we take any homomorphism $F_{X} \to F_A$ such that 
$x^{\al_1} y^{\al_1} = x^\al$, the product $x^{\al_1} y^{\al_1}$ is reduced 
and $\al_1$ coincides with $\al$ on all variables other than $x$ and $y$.

As an illustration, we prove a well known fact.

\begin{proposition}[\cite{Wicks-1972}] \label{prop:Wicks} 
Let $w \in F_A$ be a value of a coefficient-free quadratic word $Q$.
Then there exists a homomorphic image $R$ of $Q$ 
and an evaluation $\be: F_X \to F_A$ such that $x^\be\ne 1$ for all $x \in \Var(R)$, 
the word $R[\be]$ is freely reduced and equals to ~$w$.
\end{proposition}

\begin{proof}
Let $w = Q^\al$. We start with the pair $(Q,\al)$ and apply degenerate transformations and
cancellation reductions until possible. For the resulting pair $(R,\be)$, $R$ and $\be$ are as required.
\end{proof}

\begin{corollary} \label{cor:quadratic-standard-cffree}
\begin{enumerate}
\item 
If $w \in F_A$ is a value of a coefficient-free orientable quadratic word of genus $g$ then $w$
is a product of $g$ commutators $[u_1, v_1] \dots [u_g, v_g]$ with $|u_i|, |v_i| \le 2|w|$ for all $i$.
\item
If $w \in F_A$ is a value of a coefficient-free non-orientable quadratic word of genus $g$ then
$w$ can be represented as a product $u_1^2 \dots u_k^2 [u_{k+1},u_{k+2}] \dots [u_{g-1},u_g]$ 
where $|u_i| \le 2|w|$ for all $i$.
\end{enumerate}
\end{corollary}

\begin{proof}
Let $Q$ be the quadratic word from the hypothesis and
let $w = v^{-1} w_1 v$ where $w_1$ is cyclically reduced.
By Proposition \ref{prop:Wicks} and Lemma \ref{lm:cffree-reduction}, there is a standard orientable
or semi-standard non-orientable quadratic word $R$ and an evaluation $\be$ of variables in $R$
such that $R$ is a homomorphic image of $Q$, $R^\be = w_1$ and $|x^\be| \le 2|w_1|$ 
for every variable $x \in \Var(R)$. 
If needed by we add to $R$ extra commutators or squares with new variables $y$ with values $y^\be=1$ 
so that $R$ becomes equivalent to $Q$.
The required $u_i$'s and $v_i$'s are obtained by conjugating the values $x^\be$ with $v$.
\end{proof}

In what follows, we apply an elimination process to a pair $(Q,\al)$ 
where $\al$ is a solution of a standard quadratic equation written in the form
\begin{equation} \label{eq:split-equation}
  Q = z_1^{-1} c_1 z_1 z_2^{-1} c_2 z_2 \dots z_{m}^{-1} c_{m} z_{m}. 
\end{equation}
Our goal is to find a solution of this equation satisfying the bound from Theorem \ref{thm:sol-bound}.
To do this, using elimination process we find first a
short (in a certain sense) solution of an equivalent quadratic equation of the same
form, for given fixed elements $c_1,\dots,c_m \in F_A$.
Then, 
by using an ``economical'' automorphism from Section \ref{sec:quadratic-transformations} 
we reduce the equation to the ininial form thus producing the required 
``short'' solution of the original equation \eqref{eq:split-equation}. 

Instead of keeping fixed the element $Q^\al$ during transformations as in the proof of Proposition \ref{prop:Wicks}
we will keep the property that $Q^\al$ is a product of conjugates of fixed elements $c_1,\dots,c_m \in F_A$
(that is, $w = u_1^{-1} c_1 u_1 \dots u_m^{-1} c_m u_m$ for some $u_1,\dots,u_m$),
up to changing exponent signs of $c_i$ in the non-orientable case.

\begin{definition} \label{def:non-splittable}
Let $w,c_1,c_2,\dots,c_m \in F_A$ be cyclically reduced elements of $F_A$ and let
$w$ be a product of conjugates of $c_1,\dots,c_m$.
We say that $w$ is a {\em short} product of conjugates of $c_1,\dots,c_m$ if $w$,
viewed as a freely reduced cyclic word, does not have a form
$w = u h v h^{-1}$ where $h$ is non-empty and the set $\set{c_i}$ can be properly partitioned into two 
subsets
$\set{c_{p_i}}$ and $\set{c_{q_i}}$ so that $u$ is a product of conjugates of $c_{p_i}$'s and
$v$ is a product of conjugates of ~$c_{q_i}$'s.
\end{definition}

\begin{definition} (non-orientable version) \label{def:non-splittable-nonorient}
Let $w,c_1,c_2,\dots,c_m \in F_A$ be cyclically reduced elements of $F_A$.
We say that $w$ is a {\em unsigned product of conjugates of $c_1,\dots,c_m$} if
$w$ is a product of conjugates of $c_1^{\ep_1},\dots,c_m^{\ep_m}$ for some $\ep_1,\dots,\ep_m=\pm1$.

We say that $w$ is a {\em short} unsigned product of conjugates of $c_1,\dots,c_m$ if 
$w$, viewed as a freely reduced cyclic word, 
does not have a form $w = u h v h^{-1}$ where $h$ is non-empty and the set ~$\set{c_i}$ 
can be properly partitioned into two subsets $\set{c_{p_i}}$ and $\set{c_{q_i}}$ so that $u$ is a unsigned
product of conjugates of $c_{p_i}$'s and $v$ is a unsigned product of conjugates of $c_{q_i}$'s.
\end{definition}

It is not hard to prove that if $w$ is a short or short unsigned product of conjugates
of $c_1,\dots,c_m$ then then $|w| \le \sum |c_i|$. This is essentially a consequence of the van Kampen lemma,
see Lemma \ref{lm:trivial-tree-part} below.

\begin{proposition} \label{prop:main-reduction}
Let $c_1,c_2,\dots,c_m \in F_A$ and
$Q$ be a coefficient-free (orientable or non-orientable) quadratic word of genus $g$.
Suppose that the quadratic equation
$$
  Q = z_1^{-1} c_1 z_1 \dots z_{m}^{-1} c_{m} z_{m}
$$
has a solution in $F_A$.
\begin{enumerate}
\item 
If $Q$ is orientable then a short product of conjugates of $c_1$, $\dots$, $c_{m}$ 
is a product of at most $g$ commutators in $F_A$.
\item
If $Q$ is non-orientable then a short unsigned product of conjugates of $c_1$, $\dots$, $c_{m}$ 
is a product of at most $g$ squares in $F_A$.
\end{enumerate}
\end{proposition}

\begin{proof}
We consider first the case of orientable $Q$.

Let $\al$ be a solution of the equation from the hypothesis of the proposition. 
We describe a sequence of transformations starting with the pair $(Q, \ \al)$.
Any moment we will have a coefficient-free quadratic word ~$R$ and a homomorphism 
$\be: F_{X} \to F_A$ such that $R$ is a homomorphic image of $Q$ and
$R^\be$ is a product of conjugates of $c_1,\dots,c_m$. 
For the inductive parameter that will be decreased during transformations, we take the pair
$(|R^\be|, |R[\be]|)$ with lexicographic ordering. Recall that $|R^\be|$ denotes the 
length of the freely reduced word representing $R^\be$ and $|R[\be]|$ denotes the length of
the formal word $R[\be]$.

If $x^\be = 1$ for some variable $x \in \Var(R)$ then we apply the degenerate transformation 
$(x \mapsto 1)$ to $(R,\be)$ decreasing the number of variables in $R$.
The element $R^\be$ and the formal word $R[\be]$ are not changed.

If $R[\be]$ has a cancellation then we apply a cancellation
reduction so that $R^\be$ is not changed but the length of $R[\be]$ decreases.

If $R[\be]$ has a cyclic reduction then we change $\be$ by conjugating all the values $x^\be$
with the same element of $F_A$ decreasing the length of $R^\be$.
Observe also that using this operation we can change $R^\be$ to its any cyclic shift not 
increasing parameter $(|R^\be|, |R[\be]|)$.

Thus we assume that $R[\be]$ is cyclically reduced (and hence is the cyclically reduced form of ~$R^\be$) 
and all variables $x$ in $R$ have non-trivial values $x^\be \ne 1$.

Suppose that $R^\be$ is not a short product of conjugates of $c_1,\dots,c_m$, that is, 
up to a cyclic shift we have 
$R^\be = u h v h^{-1}$ where $u$, $v$ and $h\ne 1$ are as in Definition \ref{def:non-splittable}.
Splitting variables if needed we may assume that 
$h$ and $h^{-1}$ are values of single variables in $R$, that is,
$$
  R = U x_1^\ep V x_2^\de, \quad 
  U^\be = u, \ (x_1^\ep)^\be = h, \ V^\be = v, \ (x_2^\de)^\be = h^{-1}.
$$

We replace $\be$ by a new evaluation $\be_1$ 
such that the length of cyclically reduced form of $R^\be$ decreases and 
$R^{\be_1}$ is still a product of conjugates of $c_1,\dots,c_m$. After that we go back to 
the start of our procedure replacing $R^{\be_1}$ by its cyclically reduced form.
There are two cases.

{\em Case\/} 1: $x_1 = x_2$. 
We define $x^{\be_1}=1$ and leave the values
of all other variables unchanged.
Then $R^{\be_1} = (UV)^\be = uv$ and by the condition of Definition ~\ref{def:non-splittable},
$R^{\be_1}$ is a product of conjugates of $c_1,\dots,c_m$. 

{\em Case\/} 2: $x_1 \ne x_2$. Then $x_1$ occurs either in $U$ or in $V$. Suppose
that $U = U_1 x_1^{-\ep} U_2$ (the case when $x_1$ occurs in $V$ is similar). We take
$x_1^{\be_1} = ((U_2 U_1)^\ep)^\be$ and $y^{\be_1} = y^\be$ for all other $y$. 
Then 
$$
  R^{\be_1} = (U_2 U_1)^\be v h^{-1} = U_2^\be u (U_2^\be)^{-1} \cdot h v h^{-1}
$$ 
and hence $R^{\be_1}$ is a product of conjugates of $c_1,\dots,c_m$
but now we have $|R^{\be_1}| \le |R^\be| - 2|h|$.

The description of the transformation sequence is finished. 
After finitely many steps we get a pair $(R,\be)$ such that 
$R^\be$ is a short product of conjugates of $c_1,\dots,c_m$. 
Since $R$ is a homomorphic image of $Q$, it is a product of at most $g$ commutators in $F_X$.
This proves (i).

In the non-orientable case, the argument is similar with the difference that
we keep $R^\be$ being an unsigned product of conjugates of $c_1$, $\dots$, $c_m$. (In Case 2,
if $x_1$ occurs in $R$ twice with the same exponent $\ep$ then some exponent signs of $c_1,\dots,c_m$
are changed after the transformation.)
\end{proof}

\section{Unfolding Lyndon--van Kampen diagrams} \label{sec:diagrams}

Using Proposition \ref{prop:main-reduction} and automorphisms
from Section ~\ref{sec:quadratic-transformations} it is not difficult to get a bound on the size of the shortest
solution of a standard quadratic equation in $F_A$. The bound is $N n(Q) c(Q)$ 
in the case of orientable $Q$ and and $N n(Q)^2 c(Q)$ in the case of non-orientable ~$Q$. 
To improve this bound by factor $n(Q)$ we prove an extra statement. 
Informally speaking, it says that a Lyndon-van Kampen
diagram can be unfolded in an economical way.

We recall some definitions and facts about Lyndon--van Kampen diagrams
(or simply diagrams from now on for brevity).

By a diagram we mean a finite 2-dimensional cell complex $\De$ embedded in the plane $\Reals^2$ and
endowed with a labelling function $\la$ over an alphabet $Y$. The latter means that for any directed
edge $e$ of $\De$ the label $\la(e)$ is fixed which is 
either a letter in $Y^{\pm 1}$ or is empty. For any two mutually inverse directed edges $e$ and $e^{-1}$, we have
$\la(e^{-1}) = (\la(e))^{-1}$. The labelling function is naturally extended to 
paths (viewed as sequences of directed edges) in the 1-skeleton $\De^{(1)}$ of $\De$.
The label $\la(p)$ of a path $p$ is a word in $Y^{\pm1}$ which we will often identify with an element of $F_Y$.
We call the label of the boundary loop of a 2-cell $D$ of $\De$ the {\em boundary label of $\De$}, 
defined up to a cyclic shift.

We assume that all diagrams are connected and simply connected. 
We assume also that a diagram $\De$ is endowed with  
a fixed {\em base vertex} $\nu_0$ in the boundary of $\De$ and, moreover, a boundary loop
of $\De$ at $\nu_0$ is fixed. (In general, the boundary loop starting at a given boundary vertex 
may be not unique). 
The label of the fixed boundary loop of $\De$ is called the {\em boundary label of $\De$}.

We admit that the boundary label of a 2-cell $D$ of a diagram $\De$ is the empty word or has 
the form $y y^{-1}$, $y \in Y$. In this case we call $D$ a {\em trivial} 2-cell of {\em vertex}
or {\em edge} type, respectively. 
Otherwise a 2-cell is called {\em nontrivial}. The boundary label of a nontrivial 2-cell is always
assumed to be cyclically reduced.

If words $c_1$, $\dots$, $c_m$ 
are boundary labels of all nontrivial 2-cells of $\De$ then
we call $\De$ a {\em diagram over the set $\set{c_1,\dots,c_m}$}. (To be more formal,
$\set{c_1,\dots,c_m}$ should be viewed as a multiset since we admit that some $c_i$ are repeated.) 

Let $w,c_1,\dots,c_m \in F_Y$. 
A variant of the van Kampen lemma says that $w$ is a product of conjugates of $c_1,\dots,c_m$ if
and only if there is a diagram $\De$ with labelling function over ~$Y$
over the set $\set{c_1,\dots,c_m}$, with boundary label $w$. 
The proof can be found in
\cite[proof of Theorem V.1.1]{Lyndon-Schupp}. Note that 
cut off operations of spherical diagrams in the construction of ~$\De$ can be avoided
since we admit empty labels of edges and trivial 2-cells.

We need a precise description of the process of constructing a diagram from a representation
of an element $w$ as a product of conjugates of $c_1,\dots,c_m$.
We start with describing several elementary operations applied to a given diagram $\De$ which
produce a new diagram ~$\De_1$ over the same set $\set{c_1,\dots,c_m}$.

(T1) 
{\em Contracting a trivial edge.} If $e$ is an edge with distinct endpoints and
$\la(e)=1$ then we contract $e$ into a vertex. 

(T2)
{\em Contracting a trivial $2$-cell}. 
Let $D$ be a trivial 2-cell of $\De$. Assume that either $D$ has vertex type and the boundary loop 
consists of one edge or $D$ has edge type and the boundary loop of $D$ consists of exactly 2 edges
and is labelled $y y^{-1}$, $y \in Y$. 
Then we contract $D$ to a vertex or to an edge labelled $y$, respectively.

We admit also inverse operations $\rm{(T1)}^{-1}$ and $\rm{(T2)}^{-1}$, 
{\em introducing a trivial edge} and {\em introducing a trivial 2-cell},
respectively. We call operations (T1), (T2) and their inverses {\em trivial transformations}.
We introduce another two types of elementary operations which we call {\em elementary reductions}.

(R1)
{\em Folding boundary edges.} Assume that two distinct boundary directed edges ~$e_1$ and ~$e_2$ of $\De$ have 
the same label, a common initial vertex $\nu$ and distinct terminal vertices.
We assume furthermore that the path $(e_1 e_2)^{\pm1}$ occurs in the boundary loop of $\De$.
(This is not always true in the case when $\nu$ is the base vertex of $\De$.)
Then we perform folding of $e_1$ and ~$e_2$ into a single edge.

(R2)
{\em Removal of a leaf edge.} Let $e$ be an edge of $\De$ with an endpoint $\nu$
of valence 1. We assume that $\nu$ is not the base vertex of $\De$. Then we remove $e$
and its endpoint $\nu$ from $\De$. Note that $e$ does not belong to the boundary of a nontrivial 2-cell of $\De$
since boundary labels of nontrivial 2-cells are assumed to be cyclically reduced. 
Therefore, this operation
either reduces cancellation in the boundary label of $\De$ or changes a trivial 2-cell of edge type
to a trivial 2-cell of vertex type.

By definition, for all the operations introduced,
the base vertex and the boundary loop of ~$\De_1$ are inherited in the natural way from
those of $\De$. 

Observe that (T1) and (T2) do not change the boundary label $w$ of $\De$, (R1) reduces a cancellation in $w$
and (R2) either does not change $w$ or reduces a cancellation in $w$.

It is easy to see that any cancellation in the boundary label of $\De$ can be reduced by either (R1) or (R2)
after a sequence of trivial transformations $\rm{(T1)}^{\pm1}$ and $\rm{(T2)}^{\pm1}$.

\begin{definition} \label{def:diagram-folding}
We say that a diagram $\De_2$ is obtained by {\em folding} from a diagram $\De_1$ if it is the result
of application of a sequence of trivial transformations and elementary reductions 
so that the boundary label of $\De_2$
is freely reduced. We say also that $\De_1$ is obtained by {\em unfolding} from ~$\De_2$.

For technical convenience, we assume that after folding, diagram $\De_2$ always undergoes the following
{\em tightening} procedure: First, contract any trivial 2-cell whenever it can be contracted by
using (T2) and any sequence of operations $\rm{(T1)}^{\pm1}$. 
Second, contract all trivial edges whenever possible.
\end{definition}


Below we use the following properties of {\em tight} diagrams as in Definition \ref{def:diagram-folding}. 
The proof of the following lemma is an easy exercise and left to the reader.

\begin{lemma} \label{lm:normalized-diagram}
Let $\De$ be a tight diagram. Then the following assertions are true.
\begin{enumerate}
\item
All trivial $2$-cells of $\De$ have edge type. If a nontrivial edge $e$ occurs in the boundary loop
of such a $2$-cell $D$ then both $e$ and $e^{-1}$ occur in the boundary loop of $D$ (so the union of ~$D$ and 
$e$ is an annulus). In particular, no non-trivial boundary edge of $\De$ belongs to 
the boundary of a trivial 2-cell.
\item 
Any simple loop in the 1-skeleton $\De^{(1)}$ of $\De$ bounds a subdiagram which has at least one
nontrivial 2-cell.
\end{enumerate}
\end{lemma}

With a formal product 
$$
  v = u_1^{-1} c_1 u_1 \dots u_m^{-1} c_m u_m
$$ 
we associate a {\em rose diagram} $\De_0$ over the set $\set{c_1,\dots,c_m}$ 
labelled with (perhaps, not reduced) word $v$, see Fig.\ ~\ref{fig:rose}.
Using folding we can produce a new diagram $\De$ whose boundary label is the freely reduced from of ~$v$. 
(Existence of such a diagram $\De$ is essentially the main part of the van Kampen lemma.)

\begin{figure}[h]%
\begin{picture}(0,0)%
\includegraphics{fig8.pstex}%
\end{picture}%
\setlength{\unitlength}{0.00087489in}%
\begingroup\makeatletter\ifx\SetFigFont\undefined%
\gdef\SetFigFont#1#2#3#4#5{%
  \reset@font\fontsize{#1}{#2pt}%
  \fontfamily{#3}\fontseries{#4}\fontshape{#5}%
  \selectfont}%
\fi\endgroup%
{\renewcommand{\dashlinestretch}{30}%
\begin{picture}(3289,1696)(0,-10)%
\put(1590,0){\makebox(0,0)[lb]{\smash{{\SetFigFont{12}{14.4}{\rmdefault}{\mddefault}{\updefault}$\nu_0$}}}}
\put(1188,990){\makebox(0,0)[lb]{\smash{{\SetFigFont{12}{14.4}{\rmdefault}{\mddefault}{\updefault}$\ldots$}}}}
\put(15,780){\makebox(0,0)[lb]{\smash{{\SetFigFont{12}{14.4}{\rmdefault}{\mddefault}{\updefault}$c_m$}}}}
\put(1505,1320){\makebox(0,0)[lb]{\smash{{\SetFigFont{12}{14.4}{\rmdefault}{\mddefault}{\updefault}$c_2$}}}}
\put(3095,735){\makebox(0,0)[lb]{\smash{{\SetFigFont{12}{14.4}{\rmdefault}{\mddefault}{\updefault}$c_1$}}}}
\put(1050,285){\makebox(0,0)[lb]{\smash{{\SetFigFont{12}{14.4}{\rmdefault}{\mddefault}{\updefault}$u_m$}}}}
\put(920,600){\makebox(0,0)[lb]{\smash{{\SetFigFont{12}{14.4}{\rmdefault}{\mddefault}{\updefault}$\nu_m$}}}}
\put(1905,890){\makebox(0,0)[lb]{\smash{{\SetFigFont{12}{14.4}{\rmdefault}{\mddefault}{\updefault}$\nu_2$}}}}
\put(1815,600){\makebox(0,0)[lb]{\smash{{\SetFigFont{12}{14.4}{\rmdefault}{\mddefault}{\updefault}$u_2$}}}}
\put(2130,285){\makebox(0,0)[lb]{\smash{{\SetFigFont{12}{14.4}{\rmdefault}{\mddefault}{\updefault}$u_1$}}}}
\put(2400,420){\makebox(0,0)[lb]{\smash{{\SetFigFont{12}{14.4}{\rmdefault}{\mddefault}{\updefault}$\nu_1$}}}}
\end{picture}
}
\caption{}\label{fig:rose}
\end{figure}

We can consider also an inverse process: given a diagram $\De$ over a set $\set{c_1,\dots,c_m}$ 
with boundary label ~$w$ 
we can unfold $\De$ to get a rose diagram $\De_0$ and its associated factorization 
$$
  w = u_1^{-1} c_{\si(1)}' u_1 \dots u_m^{-1} c_{\si(m)}' u_m
$$
of $w$ in the free group $F_Y$
where $\si$ is a permutation on the set $\set{1,\dots,m}$ and $c_i'$ is a cyclic shift of ~$c_i$.
The following proposition describes the relationship between two factorizations obtained in this way.

\begin{proposition} \label{prop:diagram-braid-autos}
Let $\De$ be a diagram with labeling function over ~$Y$, with boundary label ~$w$.
Let $\De_1$ and $\De_2$ be two rose diagrams obtained from $\De$ by unfolding and let
$$
  w = u_1^{-1} c_1 u_1 \dots u_m^{-1} c_m u_m = 
      v_{\si(1)}^{-1} c_{\si(1)}' v_{\si(1)} \dots v_{\si(m)}^{-1} c_{\si(m)}' v_{\si(m)}
$$
be the associated factorizations of $w$ in $F_Y$ where $\si$ is a permutation on the set $\set{1,\dots,m}$ and
$c_i'$ is a cyclic shift of $c_i$.
Then there is an $F_Y$-automorphism $\phi$ of $F_{Y \cup \set{z_1,\dots,z_m}}$ such that 
$$
  (z_1^{-1} c_1 z_1 \dots z_m^{-1} c_m z_m)^\phi = 
    z_{\si(1)}^{-1} c_{\si(1)}' z_{\si(1)} \dots z_{\si(m)}^{-1} c_{\si(m)}' z_{\si(m)}
$$
and the following diagram is commutative
$$
  \xymatrix @C-2em{
    F_{Y \cup \set{z_1,\dots,z_m}} \ar[rr]^{\phi} \ar[dr]_{\al} && 
    F_{Y \cup \set{z_1,\dots,z_m}} \ar[dl]^{\be} \\
    & F_Y
  }
$$
where $\al$ and $\be$ are the corresponding evaluations of variables $z_i$:
$$
  z_i^\al = u_i, \quad z_i^\be = v_i \quad (i=1,\dots,m).
$$
\end{proposition}

\begin{proof}
Let $\nu_0$ be the base vertex of $\De_1$ and $\nu_1,\dots,\nu_m$ be vertices where disks
with boundary labels ~$c_i$ are attached to the other part of $\De_1$ (shown in Fig.\ ~\ref{fig:rose}). 
Let $z_i$ be the simple path between ~$\nu_i$ and ~$\nu_0$ and $d_i$ the boundary loop of the 
corresponding 2-cell at $\nu_i$ (which is labelled ~$c_i$).
We keep similar notations $\nu_i'$, $z_i'$ and $d_i'$ for $\De_2$. 

Unfolding transformations $\De \to \De_i$ induce
a homotopy equivalence $\psi$ between 1-skeletons $\De_1^{(1)}$ and $\De_2^{(1)}$ 
such that $\psi(\nu_0)=\psi(\nu_0')$. 
In particular, we have the induced isomorphism $\pi_1(\De_1^{(1)},\nu_0) \to \pi_1(\De_2^{(1)},\nu_0')$
of fundamental groups.
Moreover, from the construction of elementary operations it is easy to 
see that $\psi$ satisfies the following conditions:
\begin{enumerate}
\item 
For any loop $p$ at $\nu_0$, the image $\psi(p)$ is a loop at $\nu_0'$ labelled with 
a word representing the same element of $F_Y$.
\item
$\psi(z_i^{-1} d_i z_i)$ has the form $w_i^{-1} d_i' w_i$ for some $w_i$ 
up to homotopy in $\De_2^{(1)}$ rel $\nu_0'$.
\item
The image $\psi(\ell)$ of 
the boundary loop $\ell$ of $\De_1$ is homotopic in $\De_2^{(1)}$ rel $\nu_0'$
to the boundary loop of $\De_2$.
\end{enumerate}

The path $z_i'^{-1} w_i$ is a loop at $\nu_0'$, so it has an expression (up to homotopy rel $\nu_0'$)
$$
  z_i'^{-1} w_i = f_i (z_1'^{-1} d_1' z_1', \dots, z_m'^{-1} d_m' z_m'))
$$
in terms of the generators $z_i'^{-1} d_i' z_i'$ of $\pi_1(\De_2,\nu_0')$. We define formally a
homomorphism $\hat\psi: F_{\set{d_1,\dots,d_m,z_1,\dots,z_m}} \to F_{\set{d_1',\dots,d_m',z_1',\dots,z_m'}}$
by 
$$
  \hat\psi(d_i) = d_i' \quad\text{and}\quad \hat\psi(z_i) = z_i' f_i.
$$
The definition implies that for any loop $p$ at $\nu_0$, we have $\psi(p) = \hat\psi(p)$ up to homotopy
rel ~$\nu_0'$. From this we conclude that 
all $z_i'^{-1} d_i' z_i'$ belong to the image of $\hat\psi$ and hence $\hat\psi$ is in fact an isomorphism.
From (iii) we get 
$$
  z_1^{-1} d_1 z_1 \dots z_m^{-1} d_m z_m  \xrightarrow{\hat\psi}
    z_{\si(1)}'^{-1} d_{\si(1)}' z_{\si(1)}' \dots z_{\si(m)}'^{-1} d_{\si(m)}' z_{\si(m)}'
$$

Now we adjust $\hat\psi$ to get the required $\phi$. 
By (i) and (ii), we have
$$
  \la(z_i) = g_i \la(w_i) \quad\text{for some } g_i \in F_Y \text{ with } g_i^{-1} c_i g_i = c_i'
$$
We define $\phi$ by
$$
  \phi(z_i) = g_i z_i f_i (z_1^{-1} c_1' z_1, \dots, z_m^{-1} c_m' z_m))
$$
It is easy to see that $\phi$ satisfies all the required conditions.
\end{proof}

To formulate the main result of the section, we need two more definitions. Let $e$ be a nontrivial edge
of a diagram $\De$.
We call $e$ a {\em tree edge} if both $e$ and $e^{-1}$ occur in the boundary loop
of $\De$ (or, in other words, $\De$ splits into two disjoint subdiagrams after removal of $e$).
We call $e$ {\em tubular} if $e$ and $e^{-1}$ occur in the boundary loop of a trivial 2-cell of $\De$.

\begin{proposition} \label{prop:diagram-unfolding}
Let $w,c_1,c_2,\dots,c_m$ be cyclically reduced elements of a free group $F_Y$. 
Let ~$\De$ be a diagram with boundary label $w$ over the set 
$\set{c_1,\dots,c_m}$
%
and let $N$ denote the total length of labels of tree and tubular edges of $\De$.
Assume that $\De$ is tight as in Definition ~\ref{def:diagram-folding}.

Then there is rose diagram obtained from $\De$ by unfolding such that for the associated factorization
$$
  w = v_1^{-1} c_{\si(1)}' v_1 \dots v_m^{-1} c_{\si(m)}' v_m
$$
where $\si$ is a permutation on the indices $\set{1,\dots,m}$ and $c_i'$ is a cyclic shift of $c_i$,
the following is true:
\begin{enumerate}
\item 
$|v_i|\le \sum_{i=1}^m |c_i|+2N$ for all $i$.
\item
for any increasing sequence $1\le t_1 < t_2 < \dots < t_k \le m$ of indices $t_i$,
$$
  |v_{t_1}^{-1} c_{\si(t_1)}' v_{t_1} \dots v_{t_k}^{-1} c_{\si(t_k)}' v_{t_k}|\le \sum_{i=1}^m |c_i| + 2N.
$$
\end{enumerate}
\end{proposition}

\begin{proof}
%
Suppose that $e$ is an edge in the interior of $\De$ such that at least one 
of the endpoints of $e$ belongs to the boundary of $\De$. Then we can apply to $\De$ an unfolding
operation $\rm{(R1)}^{-1}$ replacing $e$ by two new boundary edges.

Starting from $\De$, we perform recursively unfoldings of all interior edges.
Then we contract all trivial 2-cells using (T1) and (T2).
The resulting diagram $\ti\De$ has no interior edges and no trivial 2-cells.
For convenience, we further introduce new trivial edges where needed, so that each vertex of $\ti\De$ 
gets valence at most 3 (see Fig. ~\ref{fig:unfolding})

\begin{figure}[h]%
\begin{picture}(0,0)%
\includegraphics{fig9.pstex}%
\end{picture}%
\setlength{\unitlength}{0.00087489in}%
\begingroup\makeatletter\ifx\SetFigFont\undefined%
\gdef\SetFigFont#1#2#3#4#5{%
  \reset@font\fontsize{#1}{#2pt}%
  \fontfamily{#3}\fontseries{#4}\fontshape{#5}%
  \selectfont}%
\fi\endgroup%
{\renewcommand{\dashlinestretch}{30}%
\begin{picture}(6053,1861)(0,-10)%
\put(4575,1185){\makebox(0,0)[lb]{\smash{{\SetFigFont{12}{14.4}{\rmdefault}{\mddefault}{\updefault}$s_3$}}}}
\put(15,128){\makebox(0,0)[lb]{\smash{{\SetFigFont{12}{14.4}{\rmdefault}{\mddefault}{\updefault}$\nu_0$}}}}
\put(2805,128){\makebox(0,0)[lb]{\smash{{\SetFigFont{12}{14.4}{\rmdefault}{\mddefault}{\updefault}$\nu_0=\xi_1$}}}}
\put(3185,558){\makebox(0,0)[lb]{\smash{{\SetFigFont{12}{14.4}{\rmdefault}{\mddefault}{\updefault}$s_1$}}}}
\put(3651,592){\makebox(0,0)[lb]{\smash{{\SetFigFont{12}{14.4}{\rmdefault}{\mddefault}{\updefault}$\mu_1$}}}}
\put(4200,843){\makebox(0,0)[lb]{\smash{{\SetFigFont{12}{14.4}{\rmdefault}{\mddefault}{\updefault}$\xi_3$}}}}
\put(4060,512){\makebox(0,0)[lb]{\smash{{\SetFigFont{12}{14.4}{\rmdefault}{\mddefault}{\updefault}$R_1$}}}}
\put(4630,559){\makebox(0,0)[lb]{\smash{{\SetFigFont{12}{14.4}{\rmdefault}{\mddefault}{\updefault}$\xi_2$}}}}
\put(5082,442){\makebox(0,0)[lb]{\smash{{\SetFigFont{12}{14.4}{\rmdefault}{\mddefault}{\updefault}$s_2$}}}}
\put(4870,1090){\makebox(0,0)[lb]{\smash{{\SetFigFont{12}{14.4}{\rmdefault}{\mddefault}{\updefault}$\mu_3$}}}}
\put(5120,1463){\makebox(0,0)[lb]{\smash{{\SetFigFont{12}{14.4}{\rmdefault}{\mddefault}{\updefault}$R_2$}}}}
\put(5700,558){\makebox(0,0)[lb]{\smash{{\SetFigFont{12}{14.4}{\rmdefault}{\mddefault}{\updefault}$R_3$}}}}
\put(2897,997){\makebox(0,0)[lb]{\smash{{\SetFigFont{12}{14.4}{\rmdefault}{\mddefault}{\updefault}$\longrightarrow$}}}}
\put(4502,15){\makebox(0,0)[lb]{\smash{{\SetFigFont{12}{14.4}{\rmdefault}{\mddefault}{\updefault}$\ti\De$}}}}
\put(5300,680){\makebox(0,0)[lb]{\smash{{\SetFigFont{12}{14.4}{\rmdefault}{\mddefault}{\updefault}$\mu_2$}}}}
\end{picture}
}
\caption{}\label{fig:unfolding}
\end{figure}

We claim that $\ti\De$ has exactly $N$ tree edges. Indeed,
since $\De$ is tight, by Lemma \ref{lm:normalized-diagram}(i) any nontrivial
interior edge of $\De$ either occurs in the boundary loop of a nontrivial 2-cell or is tubular.
Then every tree edge of $\ti\De$ comes either from a tree edge of $\De$ (which is not
changed in the transformation) or from a tubular edge $e$ of $\De$ (which becomes a tree edge after 
unfolding $e$ and contracting the trivial 2-cell $D$ with $\bd D \supset e$).

Let $\nu_1$ be the base vertex of $\ti\De$ and $\ell$ be the boundary loop of $\ti\De$.
We enumerate cells $R_1,\dots, R_m$ of ~$\ti\De$ 
in the order as their boundary edges meet first in ~$\ell$. Let
$$
  \set{v_0} = \De_0 \subset \De_1 \subset \dots \subset \De_m = \ti\De
$$
be a sequence of subdiagrams of $\ti\De$ 
where each $\De_{i}$ is obtained from $\De_{i-1}$ by attaching the topological closure of $R_i$ and 
a segment $s_i$ joining $R_i$ with $\De_{i-1}$. 
We view $s_i$ as a path from a vertex $\mu_i$ in $\bd R_i$ 
to a vertex ~$\xi_i$ in $\bd \De_{i-1}$. The boundary loop of $R_i$ at $\mu_i$ is labelled with a cyclic
shift $c_{\si(i)}'$ of $c_{\si(i)}$ where $\si$ is a permutation on $\set{1,\dots,m}$. Let $r_i$ be
the terminal segment of the boundary loop of $\De_{i-1}$ starting at $\xi_i$. Denote $v_i = \la(s_i r_i)$.
Then we have
$$
  w = v_1^{-1} c_{\si(1)}' v_1 \dots v_m^{-1} c_{\si(m)}' v_m
$$
and from the construction we can easily see that the rose diagram associated with this factorization
is obtained by unfolding from $\ti\De$ successively along paths $r_m$, $r_{m-1}$, $\dots$, $r_1$
each time slicing off one 2-cell $D_i$.

Observe that $s_i r_i$ passes through every tree edge of $\ti\De$ at most twice
and through every other edge of $\ti\De$ at most once.
This implies 
$$
  |u_i|\le \sum_{i=1}^m |c_i|+2N.
$$

Let us prove (ii).
Let $1\le t_1 < t_2 < \dots < t_k \le m$ be an increasing sequence of indices ~$t_i$. 
Let $h_i$ be the component of the intersection $\bd R_i \cap \ell$ that meets first in $\ell$.
We consider a diagram $\De'$ obtained by removal from $\ti\De$ all
2-cells $R_i$ and all arcs $h_i$ with $i\ne t_j$. 
It is easy to see that the boundary label of $\De'$ is equal to 
$\prod_{i=1}^k v_{t_i}^{-1} c_{\si(t_i)}' v_{t_i}$. 
Hence
$$
  \left| \prod_{i=1}^k v_{t_i}^{-1} c_{\si(t_i)}' v_{t_i} \right| \le  \sum_{i=1}^m |c_i| + 2N
$$
\end{proof}

For the proof of Theorem \ref{thm:sol-bound}, we will use Proposition \ref{prop:diagram-unfolding}
in the special case when $w$ is a short
product of conjugates of elements $c_1,\dots,c_m$ in the sense of Definition \ref{def:non-splittable}. 
We start with the following simple observation.

\begin{lemma} \label{lm:trivial-tree-part}
Let $w,c_1,c_2,\dots,c_m$ be cyclically reduced elements of a free group $F_Y$ and
$w$ be a short product of conjugates of $c_1,\dots,c_m$. 
Then any diagram $\De$ over $\set{c_1,\dots,c_m}$ with boundary label $w$
has no tree edges.
In particular, $|w| \le \sum_{i} |c_i|$.
\end{lemma}

\begin{proof}
If $e$ is a tree edge of $\De$ then the boundary loop of $\De$, up to a cyclic shift, has the form
$p e q e^{-1}$ where $p$ and $q$ are loops which bound two subdiagrams $\De_1$ and $\De_2$ of $\De$.
Since ~$w$ is cyclically reduced, the labels of $p$ and $q$ are nontrivial elements of $F_Y$ and hence
both ~$\De_1$ and ~$\De_2$ have nontrivial 2-cells. We get a
factorization of $w$ as in Definition \ref{def:non-splittable} which contradicts the assumption
that $w$ is short.

To prove the second statement, 
we take a tight diagram $\De$ 
over $\set{c_1,\dots,c_m}$ with boundary label $w$.
By the first statement and Lemma \ref{lm:normalized-diagram}(i),
any boundary edge of $\De$ belongs to the boundary of
a nontrivial 2-cell. The total length of their boundary labels is $\sum_{i} |c_i|$.
\end{proof}

\begin{corollary} \label{cor:cell-ordering}
Let $w,c_1,c_2,\dots,c_m$ be cyclically reduced elements of a free group $F_Y$ and let ~$w$ be 
a short product of conjugates of $c_i$'s. 
Then, up to re-enumeration and cyclic shifts of $c_i$'s, 
there exists a factorization 
\begin{equation*} 
   w = u_1^{-1} c_1 u_1 \dots u_m^{-1} c_m u_m.
\end{equation*}
such that $|u_i|\le \sum_{i=1}^m |c_i|$ and 
for any increasing sequence of indices $1\le t_1 < t_2 < \dots < t_k \le m$,
$$
  |u_{t_1}^{-1} c_{t_1} u_{t_1} \dots u_{t_k}^{-1} c_{t_k} u_{t_k}|\le \sum_{i=1}^m |c_i|.
$$
\end{corollary}

\begin{proof}
Take any diagram $\De$ over $\set{c_1,\dots,c_m}$ with boundary label $w$. 
Assume that $e$ is a tubular edge of $\De$. Let $D$ be the trivial 2-cell of 
$\De$ with boundary loop $p e q e^{-1}$ where $\la(p) = \la(q) = 1$.
One of the loops $p$ or $q$ bounds a subdiagram ~$\De_1$ of $\De$.
Then we can remove the annulus $D \cup e$ from $\De$ replacing the subdiagram $\De_1 \cup D \cup e$ by $\De_1$.

Continuing this process we may assume that $\De$ has no tubular edges.
It remains to apply Lemma ~\ref{lm:trivial-tree-part} and Proposition \ref{prop:diagram-unfolding}.
\end{proof}

\begin{proof}[Proof of Theorem \ref{thm:sol-bound}.]
Assume that a quadratic equation $Q=1$ is solvable in $F_A$. Let $\psi \in \Aut_{F_A} (F_{A\cup X})$ 
be an automorphism given by 
Proposition ~\ref{prop:standard-red-bound-orientable} or 
Proposition ~\ref{prop:semistandard-red-bound-nonorientable} such that $Q^\phi$ is
conjugate to a standard orientable or semi-standard non-orientable quadratic word $R$. If $\al$
is a solution of $R=1$ then $\psi\al$ is a solution of $Q=1$ and for any $x \in \Var(Q)$,
$$
  |x^{\psi\al}| \le 4 n(Q) \max_{x \in \Var(Q)} |x^\al| + 2 c(Q).
$$
We have $n(R) \le n(Q)$ and $c(R) \le c(Q)$ and we can always assume that $x^\al = 1$ for every
$x \in \Var(Q) \sm \Var(R)$.
This implies that the statement of the theorem for $Q$ follows
from the statement for $R$. Therefore, it is sufficient to prove the theorem in the case of
a standard or semi-standard $Q$. Let
$$
  Q = [x_{1},y_{1}] [x_{2},y_{2}] \dots [x_g, y_g] 
  c_1 z_2^{-1} c_2 z_2 \dots z_{m}^{-1} c_{m} z_{m}
$$
or
$$
  Q = x_1^2 x_2^2 \dots x_k^2 [x_{k+1},x_{k+2}] \dots [x_{g-1}, x_g] 
  c_1 z_2^{-1} c_2 z_2 \dots z_{m}^{-1} c_{m} z_{m}.
$$

We can further assume that all $c_i$ are cyclically reduced. Indeed,
in case of general $c_i$'s we consider the equation $\bar Q = 1$ where
each $c_i$ is replaced with its cyclically reduced form ${\bar c}_i = u_i^{-1} c_i u_i$.
A solution of $Q=1$ can be obtained from a solution of $\bar Q=1$ by 
the substitution 
$$
  (x_i \mapsto u_1 x_i u_1^{-1}, \ y_i \mapsto u_1 y_i u_1^{-1}, \ 
  z_i \mapsto u_i z_i u_1^{-1} \text{ for all } i).
$$
It is easy to see that the statement of the theorem for $Q=1$ follows from the statement for $\bar Q = 1$.

Now consider two cases.

{\em Case\/} 1: $Q$ is standard orientable, i.e.
$$
  Q = [x_{1},y_{1}] [x_{2},y_{2}] \dots [x_g, y_g] 
  c_1 z_2^{-1} c_2 z_2 \dots z_{m}^{-1} c_{m} z_{m}.
$$
By Proposition \ref{prop:main-reduction}, there exists a solution $\be$ of the 
equation 
$$
  [x_1,y_1] [x_2,y_2] \dots [x_g,y_g] z_1^{-1} c_1 z_1 \dots z_{m}^{-1} c_{m} z_{m} = 1
$$
such that the cyclically reduced form $w$ of $(z_1^{-1} c_1 z_1 \dots z_{m}^{-1} c_{m} z_{m})^\be$
is a short product of conjugates of $c_i$'s. In particular, $|w| \le c(Q)$ by Lemma \ref{lm:trivial-tree-part}.
%
By Corollary \ref{cor:cell-ordering},
there exist a permutation $\si$ on $\set{1,\dots,m}$ and elements $u_1,\dots,u_{m} \in F_A$
of length $|u_i| \le 2 c(Q)$ such that
$$
  w = u_1^{-1} c_{\si(1)} u_1 \dots u_{m}^{-1} c_{\si(m)} u_{m}
$$
and for any increasing sequence of indices $1\le t_1 < t_2 < \dots < t_k \le m$,
$$
  |u_{t_1}^{-1} c_{\si(t_1)} u_{t_1} \dots u_{t_k}^{-1} c_{\si(t_k)} u_{t_k}|\le c(Q).
$$
This implies by Lemma \ref{lm:permut-coefficients} that there are elements $v_1,\dots,v_{m}$ of length
$|v_i| \le 3 c(Q)$ such that
$$
  w = v_1^{-1} c_1 v_1 \dots v_{m}^{-1} c_{m} v_{m}.
$$
By Corollary \ref{cor:quadratic-standard-cffree}, $w$ can be represented as a product of $g$ 
commutators of elements of length at most $2c(Q)$. Hence we get a solution $\ga$ of the 
equation 
\begin{equation} \label{eq:eqn-with-z1}
  [x_1,y_1] [x_2,y_2] \dots [x_g,y_g] z_1^{-1} c_1 z_1 \dots z_{m}^{-1} c_{m} z_{m} = 1
\end{equation}
such that $|x_i^\ga|, |y_i^\ga| \le 2 c(Q)$ and $|z_i^\ga| \le 3 c(Q)$. 
To get a solution of the original equation $Q=1$, we eliminate $z_1$ by applying 
to the left-hand side of \eqref{eq:eqn-with-z1} the automorphism $\psi\in\Aut (F_{A\cup X})$ defined by 
$$
  \psi = (x_i \mapsto z_{1}^{-1} x_i z_{1}, \ y_i \mapsto z_{1}^{-1} y_i z_{1} \ 
  \text{for all $i$}, \ z_j \mapsto z_j z_{1} \ \text{for $j>1$}).
$$
This gives a solution $\psi^{-1}\ga$ of $Q=1$ with $|x^{\psi^{-1}\ga}| \le 8c(Q)$ for all $x$.

{\em Case\/} 2: $Q$ is semi-standard non-orientable, i.e.
$$
  Q = x_1^2 x_2^2 \dots x_k^2 [x_{k+1},x_{k+2}] \dots [x_{g-1}, x_g] 
  c_1 z_2^{-1} c_2 z_2 \dots z_{m}^{-1} c_{m} z_{m}.
$$
Similarly to Case 1, for some $\ep_1,\dots,\ep_m \in\set{-1,1}$ 
we find a solution $\ga$ of another semi-standard equation 
$$
  x_1^2 \dots x_r^2 [x_{r+1},x_{r+2}] \dots [x_{g-1},x_g] 
  z_1^{-1} c_1^{\ep_1} z_1 \dots z_{m-1}^{-1} c_{m}^{\ep_{m}} z_{m} = 1
$$
where $|x_i^\ga| \le 2c(Q)$ and $|z_i^\ga| \le 3c(Q)$ for all $i$. 
By Lemma ~\ref{lm:semistandard-transition}, there is an automorphism $\phi\in\Aut(F_{A\cup X})$
carrying its left-hand side, up to conjugation, to a word 
$$
  Q_1 = x_1^2 \dots x_k^2 [x_{k+1},x_{k+2}] \dots [x_{g-1},x_g] 
  z_1^{-1} c_1 z_1 z_2^{-1} c_2 z_2 \dots z_{m}^{-1} c_{m} z_{m}
$$
which produces a solution $\ga_1 = \phi^{-1}\ga$ of $Q_1=1$ with 
$$
  |x^{\ga_1}| \le (8g + 12m + 2) c(Q) \quad\text{for all } x.
$$
Finally, we eliminate $z_1$ as in the orientable case which at most triples the bound, and observe
that $8g + 12m +2 \le 12 n(Q)$.
\end{proof}

\section{Bounding parametric solutions} \label{sec:parametric}

In this section we prove Theorem \ref{thm:param-sol-bound}.

As defined in Section \ref{sec:intro}, by a parametric solution of an equation $E=1$ in a free group ~$F_A$
we mean an $F_A$-homomorphism $\eta: F_{A\cup\Var(E)} \to F_{A\cup T}$ such that $E^\eta=1$ where
$T$ is a set of parameters. Here instead of the ``big'' group $F_{A\cup X}$ we consider the group
$F_{A\cup\Var(E)}$ involving only variables occurring in $Q$. 
It will be convenient to change this point of view by introducing formal sets of variables for equations.
We assume that an equation $E=1$ is 
endowed with a formal finite set of variables $V \subset X$ such that $V \supseteq \Var(E)$
(in other words, we admit now fictitious variables $x \in V$ not occurring in $E$).
A parametric solution of such an equation $(E=1,V)$ is then an
$F_A$-homomorphism $\be: F_{A\cup V} \to F_{A \cup T}$ such that $E^\be = 1$.
If $V = \Var(E)$ then we get equations and their parametric solutions in the initial sense.

Transformations of equations are no longer $F_A$-automorphisms 
(or $F_A$-endomorphisms if degenerate transformations are allowed)
of the big free group $F_{A \cup X}$ but homomorphisms $\phi: F_{A \cup V} \to F_{A \cup V_1}$
where $V,V_1 \subset X$ are finite sets of variables. Since we want $\phi$ to be ``potentially invertible''
we require that $\phi$ be a monomorphism.
Moreover, we will require that the condition given by the following definition should be satisfied.

\begin{definition} \label{def:primitive-mons}
We call an $F_A$-monomorphism $\phi: F_{A \cup V} \to F_{A \cup V_1}$ {\em primitive} if the
image of ~$\phi$ is a free factor of $F_{A \cup V_1}$.
\end{definition}

The main ingredient to the proof of Theorem \ref{thm:param-sol-bound} is the following
proposition which may be viewed as an advanced form of Proposition \ref{prop:main-reduction}.

\begin{proposition} \label{prop:main-reduction-strong}
Let $\set{c_1,\dots,c_m}$ be a finite set of cyclically reduced elements of $F_A$.
Suppose that $\eta: F_{A\cup \Var(Q) \cup Z} \to F_{A\cup T}$ 
is a parametric solution of a quadratic equation in $F_A$ of the form
$$
  Q = z_1^{-1} c_1 z_1 \dots z_{m}^{-1} c_{m} z_{m} 
$$
where $Q$ is a coefficient-free quadratic word and $Z = \set{z_1,\dots,z_m}$.

Then there is a coefficient-free quadratic word $R$ equivalent to $Q$, a finite set 
of variables $V \supseteq \Var(R)$
and a parametric solution $\th: F_{A \cup V \cup Z} \to F_{A\cup T}$ of an equation
$$
  R = z_{\si(1)}^{-1} c_{\si(1)}^{\ep_1} z_{\si(1)} \dots z_{\si(m)}^{-1} c_{\si(m)}^{\ep_{m}} z_{\si(m)}
$$
where all $\ep_i =1$ if $Q$ is orientable, $\ep_i \in \set{-1,+1}$ if $Q$ is non-orientable
and $\si$ is a permutation on $\set{1,2,\dots,m}$,
such that the following assertions are true:
\begin{enumerate}
\item
There is a primitive $F_A$-monomorphism $\phi:  F_{A\cup \Var(Q) \cup Z} \to  F_{A \cup V \cup Z}$ and 
an $F_A$-endomorphism $\om \in \End_{F_A}(F_{A\cup T})$ such that 
$$
  (Q^{-1} z_1^{-1} c_1 z_1 \dots z_{m}^{-1} c_{m} z_{m})^\phi
$$
is conjugate to
$$
  R^{-1} z_{\si(1)}^{-1} c_{\si(1)}^{\ep_1} z_{\si(1)} \dots z_{\si(m)}^{-1} c_{\si(m)}^{\ep_{m}} z_{\si(m)}
$$
and the following diagram is commutative:
$$
  \xymatrix{
    F_{A\cup \Var(Q) \cup Z} \ar[r]^\phi \ar[d]_\eta &
    F_{A \cup V \cup Z} \ar[d]^\th \\
    F_{A\cup T} &
    F_{A\cup T} \ar[l]_\om
  }
$$
\item
Let $\bar R$ denote the word obtained by removing from $R$ all variables $x$ with $x^\th=1$ and 
performing all subsequent cancellations.
Then $\bar R[\th]$ is cyclically reduced (and hence $\bar R[\th]$
is the cyclically reduced form of $R^\th$; recall that $W[\th]$ denotes the formal word
obtained after substitution in $W$ of values $x^\th$ of all variables $x$).
\item
There is a Lyndon--van Kampen diagram $\De$ with boundary label $\bar R^\th$ over the set
$\set{c_1^{\ep_1}$, $\dots$, $c_{m}^{\ep_{m}}}$ 
folded from the rose diagram associated with factorization
\begin{equation} \label{eq:R-fact}
  \bar R^\th = 
	      (z_{\si(1)}^\th)^{-1} c_{\si(1)}^{\ep_1} z_{\si(1)}^\th \dots 
	      (z_{\si(m)}^\th)^{-1} c_{\si(m)}^{\ep_{m}} z_{\si(m)}^\th
\end{equation}
such that 
any tree or tubular edge of $\De$ is labelled by a sigle parameter letter and the total
number of tree and tubular edges is less than $m$.
\end{enumerate}
\end{proposition}

\begin{proof}
We consider first the case of orientable $Q$ (so all $\ep_i$ are 1).

We describe a certain transformation process.
At any moment, we will have the following data:
\begin{itemize}
\item 
A quadratic word $R^*$ of the form 
$$
  R^* = 
  R^{-1} z_{\si(1)}^{-1} c_{\si(1)} z_{\si(1)} \dots z_{\si(m)}^{-1} c_{\si(m)} z_{\si(m)}
$$
where $R$ is a coefficient-free quadratic word equivalent to $Q$;
\item
A finite set $V$ of variables such that $V \supseteq \Var(R)$.
\item
A parametric solution $\th: F_{A\cup V \cup Z} \to F_{A\cup T}$ of the equation $R^*=1$.
\item 
A primitive $F_A$-monomorphism $\phi: F_{A \cup \Var(Q) \cup Z} \to F_{A \cup V \cup Z}$ and an endomorphism
$\om \in \End_{F_A}(F_{A\cup T})$ satisfying (i).
\end{itemize}

We start with 
$
  R_0^* = Q^{-1} z_1^{-1} c_1 z_1 \dots z_{m}^{-1} c_{m} z_{m}
$, 
$V_0 = \Var(Q)$ and $\th_0 = \eta$ given by the hypothesis of the proposition.
For $\phi$ and $\om$, we take the identity maps.

A principal distinction from the proof of Proposition \ref{prop:main-reduction} is that we cannot
use degenerate transformations now. We can have variables $x$ with $x^\th=1$ which we 
call {\em degenerate}. Instead of $|R[\th]|$ as one of the inductive parameters, we
use $|\bar R[\th]|$ where $\bar R$ is obtained by removal of all degenerate variables from $R$ and
performing subsequent cancellations. Variables $x \in \Var(R) \sm \Var(\bar R)$
which are not degenerate are called {\em cancelled}.

Our inductive parameter now is the pair $(|R^\th|, |\bar R[\th]|)$ with 
lexicographic ordering. 
The whole transformation process consists of steps 1--5 described below. 

We start with describing several elementary transformations of triples $(R^*,V,\th)$ of the described form.
There will be two types of them. 

A {\em substitution} is given by a new set of 
variables $V_1 \supseteq V$,
an $F_A$-monomorphism $\psi: F_{A\cup V\cup Z} \to F_{A\cup V_1\cup Z}$ 
and a homomorphism $\th_1: F_{A \cup V_1 \cup Z} \to F_{A\cup T}$ such that $\th = \psi \th_1$.
The new quadratic word $R_1^*$ is defined as the cyclically reduced form of $(R^*)^\psi$.
In most cases when substitutions are defined, we change only $R$ and do not change the coefficient part 
$z_{\si(1)}^{-1} c_{\si(1)}^{\ep_1} z_{\si(1)} \dots z_{\si(m)}^{-1} c_{\si(m)}^{\ep_{m}} z_{\si(m)}$.

A {\em generalization} is given by an endomorphism $\tau\in\End_{F_A}(F_{A\cup T})$ and
a new parametric solution $\th_1$ of $R^*=1$ such that $\th = \th_1 \tau$.
In this case we get a new triple $(R_1^*,V_1,\th_1)$ where $R_1^* = R^*$ and $V_1 = V$.

Note that in both cases existence of $\phi$ and $\om$ satisfying (i) for a triple $(R^*,V,\th)$
automatically implies one for the new triple $(R_1^*,V_1,\th_1)$. So we will not care about
condition ~(i).

{\em Transferring cancelled subwords.} Application condition: 
a word of the form $x^\ep D$ occurs in ~$R$ where $x \in \Var(\bar R)$ 
and $D$ disappears in $\bar R$. We apply the substitution $\psi = (x^\ep \to x^\ep D^{-1})$ to $R$ and
do not change ~$V$ and ~$\th$. 
The transformation transfers $D$ to another location in $R$. (Observe that $x^\ep D x^{-\ep}$ cannot occur in $R$ 
since $x$ would be cancelled otherwise.)

{\em Cancellation reduction}. 
Application condition: a non-trivial cancellation in $\bar R[\th]$ occurs between 
the values of two neighboring variables $x^\ep$ and $y^\de$ ($\ep,\de=\pm1$) and
$x^\ep y^\de$ occurs also in ~$R$.
If $x \ne y$ then for some $u, v$ and $w\ne 1$,
$$
  (x^\ep)^\th = uw \quad\text{and}\quad (y^\de)^\th = w^{-1} v
$$
We introduce a new variable $z \notin V$, take $V_1 = V \cup \set{z}$ and define $\psi$ and $\th_1$ by
$$
  \psi = (x^\ep \mapsto x^\ep z, \ y^\de \mapsto z^{-1} y^\de)
$$
and
$$
  (x^\ep)^{\th_1} = u, \quad (y^\de)^{\th_1} = v, \quad z^{\th_1} = w \quad\text{and}\quad
  h^{\th_1} = h^\th \text{ for } h \ne x,y,z.
$$
The case $x=y$ is treated in a similar way (see the proof of Proposition \ref{prop:main-reduction}).

Observe that both operations do not change $R^\th$,
transferring cancelled subwords does not change also $\bar R[\th]$ and cancellation reduction
decreases the length of $\bar R[\th]$ by $2|w|$ where $w$ is the cancellable part.

{\em Step \/} 1.
Assume that $\bar R[\th]$ has a cancellation between the values of two 
neighboring variables $x^\ep$ and $y^\de$. Then a word of the form
$x^\ep D y^\de$ occurs in $R$ where $D$ disappears in ~$\bar R$. 
We transfer ~$D$ to another location 
so that $x^\ep$ and $y^\de$ become neighbors in $R$ and then apply cancellation reduction
decreasing $|\bar R[\th]|$. We repeat the procedure until 
$\bar R[\th]$ becomes freely reduced.

To make $\bar R[\th]$ cyclically reduced, we use one more transformation. 

{\em Conjugation}. Take a new variable $y \notin V$, take $V_1 = V \cup \set{y}$ and define
$\psi: F_{A \cup V \cup Z} \to F_{A \cup V_1 \cup Z}$ by
$$
  \psi = (x \mapsto y^{-1} x y \text{ for } x \in \Var(R), \ z_i \to z_i y \text{ for } i=1,\dots,m)
$$
To define $\th_1$, we take any element $u \in F_{A\cup T}$ for the value $y^{\th_1}$ of $y$ and set 
according to equality $\th = \psi \th_1$:
$$
  x^{\th_1} = u x^\th u^{-1} \text{ for } x \in \Var(R) \quad\text{and}\quad 
  z_i^{\th_1} = z_i^{\th} u, \ i=1,\dots,m.
$$
For this transformation, we have $R_1^{\th_1} = u^{-1} R^\th u$.

{\em Step \/} 2.
If $R^\th$ is not cyclically reduced then using conjugation we replace it with its 
cyclically reduced form. Then jump back to Step 1.

We can assume now that $R$ and $\th$ satisfy condition (ii).
In the rest of the proof, we show how to achieve (iii). 
We observe for the future that using conjugation we can replace $\bar R[\th]$ 
with any its cyclic shift not increasing the inductive parameter $(|R^\th|, |\bar R[\th]|)$.
We introduce yet another transformation.

{\em Splitting a variable.} Let $x \in\Var(\bar R)$.
By the definition of $\bar R$ we have $x^\th \ne 1$.
Take a new variable $y \notin V$, take $V_1 = V \cup \set{y}$ and apply to $R$
the substitution $\psi = (x \mapsto xy)$.
To define ~$\th_1$, we take any values $x^{\th_1}$ and $y^{\th_1}$ such that the product
$x^{\th_1} y^{\th_1}$ is reduced and equals ~$x^\th$. The values of all other variables
are unchanged.

Starting from now we fix any diagram $\De$ with boundary label $\bar R[\th]$ over the set $\set{c_1,\dots,c_m}$
folded from the rose diagram associated with factorization \eqref{eq:R-fact}. 
Transformations below will include also change of $\De$.

We introduce a transformation which changes the coefficient part.

{\em Rearranging coefficients.} Let $\De_0$ be a rose diagram obtained from $\De$ by unfolding, and let 
$$
  \bar R[\th] = v_{\si(1)}^{-1} c_{\si(1)} v_{\si(1)} \dots v_{\si(m)}^{-1} c_{\si(m)} v_{\si(m)}
$$
be the associated factorization of $\bar R[\th]$ into a product of conjugates of $c_i$'s.
(Note that we can always assume that a cyclic shift of $c_i$ coincides with $c_i$ in this
factorization, by performing extra unfolding operations on the rose diagram.)

By Lemma \ref{prop:diagram-braid-autos}, there 
is an automorphism $\psi \in \Aut(F_{A \cup V \cup Z})$ changing only variables in $Z$ such that
$$
  z_1^{-1} c_1 z_1 \dots z_{m}^{-1} c_{m} z_{m} \xrightarrow{\psi}
  z_{\si(1)}^{-1} c_{\si(1)} z_{\si(1)} \dots z_{\si(m)}^{-1} c_{\si(m)} z_{\si(m)}
$$
and for $\th_1 = \psi^{-1} \th$ we have $z_i^{\th_1} = v_i$. We apply the substitution $\psi$ to get a new
triple $(R^*,V,\th_1)$ where the solution $\th$ and the coefficient part in $R^*$
are only changed.

{\em Step \/} 3.
Let $e$ be a tree edge of $\De$ and let $h = \la(e) \ne 1$.
The edge $e$ divides ~$\De$ into the union 
$\De = \De_1 \cup e \cup \De_2$ of $e$ and two subdiagrams ~$\De_1$ and ~$\De_2$.
Passing to a cyclic shift of $\bar R[\th]$ if needed (which can be performed using conjugation) 
we assume that 
$$
  \bar R[\th] = u h v h^{-1}
$$ 
where occurrences of $h$ and $h^{-1}$ are the labels of ~$e$ and ~$e^{-1}$, respectively.
Splitting variables if necessary we may assume that $h$ and $h^{-1}$ are values of single variables, that is,
$$
  \bar R = U x_1 V x_2^\de, \quad 
  U^\th = u, \quad x_1^\th = h, \quad V^\th = v \quad\text{and}\quad (x_2^\de)^\th = h^{-1}.
$$
Note that $U$ and $V$ are labels of boundary loops
of subdiagrams $\De_1$ and $\De_2$. The set of 2-cells of $\De$ is partitioned into the set of 
2-cells in $\De_1$ and in $\De_2$. Let
$$
  U^\th = w_{i_1}^{-1} c_{i_1} w_{i_1} \dots w_{i_k}^{-1} c_{i_k} w_{i_k}
$$
and
$$
  V^\th = w_{i_{k+1}}^{-1} c_{i_{k+1}} w_{i_{k+1}} \dots w_{i_m}^{-1} c_{i_m} w_{i_m}
$$
where 
$$
  \set{1,2,\dots,m} = \set{i_1,\dots,i_k} \uplus \set{i_{k+1},\dots,i_m}
$$
be factorizations obtained from unfoldings of ~$\De_1$ and ~$\De_2$.
Then 
$$
  \bar R[\th] = (w_{i_1}^{-1} c_{i_1} w_{i_1}) \dots (w_{i_k}^{-1} c_{i_k} w_{i_k}) 
    (h w_{i_{k+1}}^{-1} c_{i_{k+1}} w_{i_{k+1}} h^{-1}) \dots (h w_{i_m}^{-1} c_{i_m} w_{i_m} h^{-1}).
$$
is a factorization associated to an unfolding of $\De$.
We apply first rearrangement of coefficients so that equation $R^*=1$ gets the form
$$
  R = E F \quad\text{where}\quad 
  E = z_{i_1}^{-1} c_{i_1} z_{i_1} \dots z_{i_k}^{-1} c_{i_k} z_{i_k}, \quad
  F = z_{i_{k+1}}^{-1} c_{i_{k+1}} z_{i_{k+1}} \dots z_{i_m}^{-1} c_{i_m} z_{i_m}
$$
and 
$$
  E^\th = U^\th = u, \quad F^\th = h V^\th h^{-1} = hvh^{-1}.
$$
Consider two cases.

{\em Case\/} 1: $x_1=x_2$. Since $x_1^\th = (x_2^{-\de})^\th \ne 1$ we have $\de=-1$.

We apply a generalization transformation which replaces $h$ with a new parameter $t$.
The new parametric solution $\th_1$ is defined by
$$
  \th_1: \begin{cases} 
         x_1 \mapsto t \\
	  z_j \mapsto z_j^\th h t^{-1} & \text{for } j=i_{k+1}, \dots, i_m \\
	  y \mapsto y & \text{for } y \in \Var(R) \sm \set{x_1} \text{ and } y = z_{i_1},\dots,z_{i_k}
         \end{cases}
$$
In $\De$, we change the label of $e$ to $t$.
Note that we do not change the inductive parameter $(|R^\th|, |\bar R[\th]|)$
since $h$ is the label of a single edge $e$ and hence $|h|=1$.

After performing the transformation we start a new iteration of Step 3 checking for another tree edge of $\De$. 
In general, after splittings of variables new tree edges may appear in $\De$. 
However, the total length of labels of tree edges is not changed. Therefore, after finitely 
many iteration steps we either process all tree edges of $\De$ or come to Case 2 where the inductive parameter
decreases.

{\em Case\/} 2: $x_1 \ne x_2$. Then $x_1$ occurs either in $U$ or in $V$. Without loss of
generality we assume that $x_1$ occurs in $U$ (if $x_1$ occurs in $V$ then we can pass 
to a cyclic shift of $\bar R[\th]$ by conjugation and then come to a symmetric situation). 
Let $U = U_1 x_1^{-1} U_2$.

First we apply the substitution
$$
  \psi_1 = (x_1 \mapsto x_1 E^{-1} U_1).
$$
We get 
\begin{align*}
  R^{-1} E F
    &= x_2^{-\de} V^{-1} x_1^{-1} U_2^{-1} x_1 U_1^{-1} E F \\
    &\xrightarrow{\psi_1}
      x_2^{-\de} V^{-1} U_1^{-1} 
      E x_1^{-1} U_2^{-1} x_1 F
\end{align*}
For the new value $x_1^{\th_1}$ of $x_1$ we have
$$
  x_1^{\th_1} = (x_1 U_1^{-1} E)^\th 
    = U_2^\th
$$
Next we apply another substitution 
$$
  \psi_2 = (z_{i_j} \mapsto z_{i_j} x_1^{-1} U_2 x_1, \ j=1,\dots,k)
$$
to get
$$
      x_2^{-\de} V^{-1} U_1^{-1} 
      E x_1^{-1} U_2^{-1} x_1 F 
  \xrightarrow{\psi_2}
  x_2^{-\de} V^{-1} U_1^{-1} x_1^{-1} U_2^{-1} x_1 E F .
$$
After application of $\psi_1$ and $\psi_2$ we get a new triple $(R_1^*,V,\th_1)$ with
$R_1 = x_1^{-1} U_2 x_1 U_1 V x_2^\de$ and
$$
  R_1^{\th_1} = (U_2 U_1 V x_2^\de)^\th
$$
which implies 
$$
  |R_1^{\th_1}| < |R^\th|.
$$
We jump next to Step 1. Step 3 is finished.

At this point, we produce
a triple $(R^*,V,\th)$ satisfying (ii) and a diagram ~$\De$ 
with boundary label $\bar R[\th]$ over the set $\set{c_1,\dots,c_m}$ 
folded from the rose diagram associated with \eqref{eq:R-fact}.
The diagram $\De$ has the property that
the label of any its tree edge $e$ is a single parameter letter ~$t$. 
Moreover, there are exactly two occurrences of $t$ in $\bar R[\th]$ and both
are the values of one variable $x \in \Var(\bar R)$.

{\em Step \/} 4. Let $e$ be a tubular edge of $\De$. Let $D$ be the trivial 2-cell with
boundary loop $e p e^{-1} q$ where $p$ and $q$ are loops with empty labels and $p$ bounds
a subdiagram $\De_1$ of $\De$. Without loss of generality we assume that there are no tubular
edges in $\De_1$. 
We choose any non self-intersecting path $s$ 
joining the base vertex $\nu_0$ with the start of $e$ and unfold $\De$
along the path $se p e^{-1} s^{-1}$ as shown in Fig.~\ref{fig:tubular-unfolding} (note that ~$D$ is
contracted into an edge after this procedure). Then we perform folding of the resulting diagram 
reducing cancellation in the newly appeared copies $s_1e_1 p_1 e_1^{-1} s_1^{-1}$ and 
$s_2e_2 p_2 e_3^{-1} s_3^{-1}$ of $se p e^{-1} s^{-1}$ (Fig.~\ref{fig:tubular-unfolding}).
The resulting diagram $\De'$ is the union of two subdiagrams $\Th$ and $\De_2$ with 
$\Th \cap \De_2 = \set{\nu_0}$. The subdiagram $\Th$ is obtained from the union of $\De_1$ and 
$s_1e_1$ by adding trivial 2-cells so that the complement $\Th-\De$ consists of annuli formed
by the trivial 2-cells and tubular edges that become all edges of the path $s_1 e_1$.
The subdiagram $\De_2$ is obtained from the complement $\De-\De_1$ by contracting the path $q$ into
a vertex (to do the contraction, we introduce trivial edges where needed to remove self-intersections in $q$).

Observe that the boundary label of $\Th$ is empty and the boundary label of $\De'$ is equal to 
the boundary label of $\De$. Since both $\De$ and $\De'$ can be unfolded form a common rose
diagram ~$\De_0$ we can perform rearranging coefficients so that factorization \eqref{eq:R-fact}
is replaced by the factorization associated to $\De_0$. We replace also $\De$ with $\De'$.

Similar to Step 3 we further replace occurrences of the label of $s_1e_1$ in the values $x^\be$
of variables $x$ by a new parameter letter (only values of variables $z_i$ are changed for which
the corresponding coefficient $c_i$ is the boundary label of a 2-cell in $\De_1$).

We perform the described procedure for all tubular edges of $\De$.
After this, $\De$ becomes the union of subdiagrams $\hat\De$ and $\Th_1,\dots,\Th_r$ 
with a common vertex $\nu_0$
and having no other intersections. Each subdiagram $\Th_i$ has empty label and a single tubular edge labelled by
a parameter letter. It is not hard to see that after the whole procedure no tree edges appear in $\De$.

\begin{figure}[h]%
\begin{picture}(0,0)%
\includegraphics{fig10.pstex}%
\end{picture}%
\setlength{\unitlength}{0.00087489in}%
\begingroup\makeatletter\ifx\SetFigFont\undefined%
\gdef\SetFigFont#1#2#3#4#5{%
  \reset@font\fontsize{#1}{#2pt}%
  \fontfamily{#3}\fontseries{#4}\fontshape{#5}%
  \selectfont}%
\fi\endgroup%
{\renewcommand{\dashlinestretch}{30}%
\begin{picture}(6366,2319)(0,-10)%
\put(864,1800){\makebox(0,0)[lb]{\smash{{\SetFigFont{12}{14.4}{\rmdefault}{\mddefault}{\updefault}$q$}}}}
\put(864,1485){\makebox(0,0)[lb]{\smash{{\SetFigFont{12}{14.4}{\rmdefault}{\mddefault}{\updefault}$p$}}}}
\put(842,1120){\makebox(0,0)[lb]{\smash{{\SetFigFont{12}{14.4}{\rmdefault}{\mddefault}{\updefault}$\De_1$}}}}
\put(924,735){\makebox(0,0)[lb]{\smash{{\SetFigFont{12}{14.4}{\rmdefault}{\mddefault}{\updefault}$e$}}}}
\put(1000,390){\makebox(0,0)[lb]{\smash{{\SetFigFont{12}{14.4}{\rmdefault}{\mddefault}{\updefault}$s$}}}}
\put(857,90){\makebox(0,0)[lb]{\smash{{\SetFigFont{12}{14.4}{\rmdefault}{\mddefault}{\updefault}$\nu_0$}}}}
\put(2372,390){\makebox(0,0)[lb]{\smash{{\SetFigFont{12}{14.4}{\rmdefault}{\mddefault}{\updefault}$\De_1$}}}}
\put(3570,1570){\makebox(0,0)[lb]{\smash{{\SetFigFont{12}{14.4}{\rmdefault}{\mddefault}{\updefault}$p_2$}}}}
\put(3241,430){\makebox(0,0)[lb]{\smash{{\SetFigFont{12}{14.4}{\rmdefault}{\mddefault}{\updefault}$s_2e_2$}}}}
\put(3797,430){\makebox(0,0)[lb]{\smash{{\SetFigFont{12}{14.4}{\rmdefault}{\mddefault}{\updefault}$s_3e_3$}}}}
\put(2920,30){\makebox(0,0)[lb]{\smash{{\SetFigFont{12}{14.4}{\rmdefault}{\mddefault}{\updefault}$s_1e_1$}}}}
\put(2154,1110){\makebox(0,0)[lb]{\smash{{\SetFigFont{12}{14.4}{\rmdefault}{\mddefault}{\updefault}$\longrightarrow$}}}}
\put(4929,1110){\makebox(0,0)[lb]{\smash{{\SetFigFont{12}{14.4}{\rmdefault}{\mddefault}{\updefault}$\longrightarrow$}}}}
\put(5560,1725){\makebox(0,0)[lb]{\smash{{\SetFigFont{12}{14.4}{\rmdefault}{\mddefault}{\updefault}$\De_2$}}}}
\put(5740,770){\makebox(0,0)[lb]{\smash{{\SetFigFont{12}{14.4}{\rmdefault}{\mddefault}{\updefault}$\Th$}}}}
\put(5810,428){\makebox(0,0)[lb]{\smash{{\SetFigFont{12}{14.4}{\rmdefault}{\mddefault}{\updefault}$\De_1$}}}}
\end{picture}
}
\caption{}\label{fig:tubular-unfolding}
\end{figure}

{\em Step \/} 5. The final procedure is 
elimination of tree vertices of $\De$, that is,
vertices that do not belong to the boundary of any 2-cell of $\De$.

Let $\nu$ be a tree vertex of $\De$. 
Since $\bar R[\th]$ is cyclically reduced, the valence of $\nu$ is at least ~2.
We assume without loss of generality that $\nu$ is distinct from the base vertex $\nu_0$ of $\De$
(otherwise using conjugation we can move $\nu_0$ to any non-tree boundary vertex of $\De$).

Let $e_1,\dots,e_d$ be all directed edges starting at $\nu$, and 
let $x_1,\dots,x_d$ be the corresponding variables, so $x_i^\th = t_i$ is a parameter letter which
is the label of $e_i$. 
Since $\nu\ne\nu_0$, all occurrences of variables ~$x_i$ in $\bar R$ are of the form 
 $\dots x_i^{-1} x_{i+1} \dots $ $(i \bmod d)$. Hence all occurrences of $x_i$ in ~$R$ have
the form $\dots x_i^{-1} D_i x_{i+1}\dots$ $(i \bmod d)$ where all $D_i$ disappear
in $\bar R$. We first apply to ~$R$ the substitution
$$
  (x_2 \mapsto D_1^{-1} x_1 x_2, \ x_3 \mapsto D_2^{-1} D_1^{-1} x_1 x_3, \ \dots, \
    x_k \mapsto D_k^{-1} \dots D_1^{-1} x_1 x_k)
$$ 
This eliminates $x_2$ from $\bar R$ (it either disappears in $R$ or becomes cancelled in $\bar R$).
For the new solution $\th_1$ we get $x_i^{\th_1} = t_1 t_i$ for $i=2,\dots,d$. 
All occurrences of parameter letters $t_i$ in the values of other variables (which can be 
only $z_j$'s) are of the form $\dots t_i^{-1} t_{i+1} \dots $ $(i \bmod d)$.
Hence we can perform a generalization replacing each $t_1 t_i$ with a single parameter $t_i$. 

After performing this operation, the length of $\bar R[\th]$ is decreased.
We get also a new diagram ~$\De_1$ which is obtained from $\De$ by contracting the edge $e_1$ into a vertex.

We repeat the procedure until we get rid off all tree vertices of $\De$.

The description of the transformation process is finished. Let $\De$ be a diagram obtained
after all steps 1--5. To prove the proposition, we 
have only to estimate the total number of tree and tubular edges of $\De$.

Let $e_1$, $\dots$, $e_k$ be all tree edges of $\De$ and let $r$ be the number of subdiagrams $\Th_i$
with empty label produced at Step 4. Each of the subdiagrams $\Th_i$ has at least one non-trivial
2-cell since otherwise it should be contracted to a vertex. The same is true for all 
connected components of the complement $\De - \cup_i e_i$.  This implies that the number
of these connected components is at most $m-r$. On the other hand, this number is 
precisely $k+1$ since each ~$e_i$ joins two components of $\De - \cup_i e_i$ and $\De$ 
is simply connected. Hence $k+r < m$.

The proof of Proposition \ref{prop:main-reduction-strong} is completed in the case when
$Q$ is orientable.
In the case of non-orientable $Q$ the only difference in the argument is that we admit 
inerses of coefficient factors. (In Case 2 at Step 3, it may happen
that $x_1$ occurs in $R$ twice with the same exponent $+1$ or $-1$. 
Then the corresponding transformation inverses ~$E$ which occurs in the coefficient part.)
\end{proof}

\begin{proof}[Proof of Theorem \ref{thm:param-sol-bound}]

We start with the case when $Q$ is a standard orientable quadratic word.

We assume that $Q$ has at least one coefficient. If $Q$ is coefficient-free
then there exists a parametric solution $\eta_0\in\Hom(F_{\Var(Q)},F_T)$ of the equation $Q=1$ such that
the value of each variable is either a parameter letter or trivial
and any ordinary solution $\al$ of $Q=1$ in {\em any} coefficient group $F_A$ may be
represented as $\phi\al\om$ where $\phi\in \Stab(Q)$ and $\om\in\Hom(F_T,F_A)$
(see Theorem 4 in Section 5 of a survey \cite{Grigorchuk-Kurchanov-1993}).
This implies that $\eta_0$ is a generalization of any other parametric parametric solution $\eta$ 
of $Q=1$ (since $\eta$ is an ordinary solution in $F_{A \cup T}$).
In this case, the statement of Theorem ~\ref{thm:param-sol-bound} holds with bound
$|x^{\eta_0}| \le 1$. 

Let $\eta$ be a parametric solution of a standard quadratic equation $Q=1$ where
$$
  Q =  [x_1, y_1] [x_2, y_2] \dots [x_g, y_g] c_1 z_2^{-1} c_2 z_2 \dots z_m^{-1} c_m z_m
$$
Similarly to the proof of Theorem \ref{thm:sol-bound} we may assume that all $c_i$ are cyclically reduced.
We adjust $Q$ and $\eta$ by introducing an extra variable $z_1$ with $z_1^\eta = 1$
and replacing $c_1$ in $Q$ with $z_1 ^{-1} c_1 z_1$.
By Proposition ~\ref{prop:main-reduction-strong} with $Q := ([x_1, y_1] \dots [x_g, y_g])^{-1}$
we find an equation of the form 
$$
  R = z_{\si(1)}^{-1} c_{\si(1)}^{} z_{\si(1)} \dots z_{\si(m)}^{-1} c_{\si(m)}^{} z_{\si(ml)},
$$
a finite set of variables $V \supseteq \Var(R)$
and a parametric solution $\th: F_{A \cup V \cup Z} \to F_{A\cup T}$ 
which satisfy conditions (i)--(iii) of that proposition.

Let a primitive $F_A$-monomorphism $\phi:  F_{A\cup \Var(Q_0) \cup Z} \to  F_{A \cup V \cup Z}$ 
and $\om \in \End_{F_A}(F_{A\cup T})$ be as in (i), that is, up to conjugation we have
\begin{equation} \label{eqn:param-transforms}
  Q
    \xrightarrow{\phi} 
  R^{-1} z_{\si(1)}^{-1} c_{\si(1)} z_{\si(1)} \dots z_{\si(m)} c_{\si(m)} z_{\si(m)}
\end{equation}
and the following diagram is commutative:
\begin{equation} \label{eq:param-diagram}
  \xymatrix{
    F_{A\cup \Var(Q_0) \cup Z} \ar[r]^\phi \ar[d]_\eta &
    F_{A \cup V \cup Z} \ar[d]^\th \\
    F_{A\cup T} &
    F_{A\cup T} \ar[l]_\om
  }
\end{equation}
Our strategy is to further transform the equation
$R = z_{\si(1)}^{-1} c_{\si(1)} z_{\si(1)} \dots z_{\si(m)}^{-1} c_{\si(m)} z_{\si(m)}$
and its parametric solution $\th$ so that we come back to the initial equation 
$Q = 1$. After that, with
a slight adjustment, $\th$ will give a desired generalization of $\eta$. 
We will find an ``ecomonic'' transformation which, together with conditions (ii) and (iii) of 
Proposition ~\ref{prop:main-reduction-strong}, will provide the required bound on
the size of the resulting generalization of $\eta$.

As in the proof of Proposition \ref{prop:main-reduction-strong}, each moment
we have the following data:
\begin{itemize}
\item 
A quadratic word $R^*$ of the form 
$$
  R^* = 
  R^{-1} z_{\si(1)}^{-1} c_{\si(1)} z_{\si(1)} \dots z_{\si(m)}^{-1} c_{\si(m)} z_{\si(m)}
$$
where $R$ is a coefficient-free quadratic word equivalent to $[x_1, y_1] \dots [x_g, y_g]$;
\item
A parametric solution $\th: F_{A\cup V \cup Z} \to F_{A\cup T}$ of the equation $R^*=1$.
\item 
A primitive $F_A$-monomorphism $\phi: F_{A \cup \Var(Q) \cup Z} \to F_{A \cup V \cup Z}$ and an endomorphism
$\om \in \End_{F_A}(F_{A\cup T})$ with \eqref{eqn:param-transforms} and \eqref{eq:param-diagram}.
\end{itemize}
Note that we do not include the set $V$ of formal variables here 
since we do not need to introduce new variables and $V$ will not change.

In a similar way, we will use two types of elementary transformations: substitutions and generalizations.
During the transformation process, condition (i) of Proposition ~\ref{prop:main-reduction-strong}
will be automatically held by construction.

We start now with the pair $(R^*,\th)$ obtained after application 
of Proposition ~\ref{prop:main-reduction-strong}.
Observe that conditions (ii) and ~(iii) imply the following bound on the total length of values $x^\th$ of 
variables in $x \in \Var(\bar R)$:
$$
  \sum_{x \in \Var(\bar R)} |x^\th| \le \frac12 c(Q) + m.
$$

{\em Step\/} 1: {\em Transforming the coefficient part.}
Let $\De$ be the diagram satisfying condition (iii) of Proposition ~\ref{prop:main-reduction-strong}.
We unfold $\De$ using Proposition \ref{prop:diagram-unfolding} and in a similar way as in the proof
of Theorem ~\ref{thm:sol-bound} we change the coefficient
part of equation $R^*=1$ and the parametric solution ~$\th$ using Proposition \ref{prop:diagram-braid-autos} 
so that the equation gets
the form 
$$
  R = z_1^{-1} c_1 z_1 \dots z_m^{-1} c_m z_m
$$
and we have
$$
  |z_i^\th| \le 3 c(Q) + 4m \quad\text{for all } i=1,2,\dots,m.
$$

{\em Step\/} 2: {\em Transforming $\bar R$ to the standard form. }
We apply the procedure described in the proof of Proposition \ref{prop:semistandard-red-bound-inverse} 
to make $\bar R$ a product of commutators. 
Since our transformation should apply to $R$ we mimic application of related
Nielsen automorphisms to $\bar R$ as application of automorphisms to $R$ in the following way.

Assume that $x^\ep y^\de$ occurs in $\bar R$ an we want
to apply to $\bar R$ a related Nielsen automorphism $\rho=(x^\ep \to x^\ep y^{-\de})$. 
There is a subword $x^\ep W y^\de$ of $R$ where $W$ is deleted in $\bar R$. In particular, $W^\th = 1$. 
Then application of an automorphism $(x^\ep \to x^\ep y^{-\de} W^{-1})$ to ~$R$
produces a new word ~$R_1$ such that ${\bar R}_1 = {\bar R}^\rho$.

At this step, 
we change the values of parametric solution $\th$ on variables of ~$x \in \Var(\bar R)$ only
(and will not change them until the final Step 4). According to
Lemma \ref{lm:cffree-reduction} we get
$$
  |x^\th| \le 2 c(Q) + 4m \quad \text{for all } x \in \Var(\bar R).
$$

{\em Step\/} 3: {\em Transforming the deleted part to the standard form.}
Recall that variables in $\Var(R) \sm \Var(\bar R)$ are divided into two types: 
{\em degenerate} variables $x$ with $x^\th=1$ and {\em cancelled} ones which cancel in $R$ after
removal of degenerate variables. 

Let
$$
  R = W_0 x_1^{\ep_1} W_1 x_2^{\ep_2} \dots W_{k-1} x_k^{\ep_k} W_k
$$
where $\bar R = x_1^{\ep_1} x_2^{\ep_2} \dots x_k^{\ep_k}$ and $W_i$ are 
deleted in $\bar R$. Using substitutions $(x_i^{\ep_i} \to x_i^{\ep_i} W_i^{-1})$
and $(x_{i+1}^{\ep_{i+1}} \to W_{i}^{-1} x_{i+1}^{\ep_{i+1}})$ we can move $W_i$'s not changing $\bar R$.
Since $\bar R$ is a product of commutators, it is easy to see that
we can collect all ~$W_i$'s together at any location between $x_i^{\ep_i}$ and $x_{i+1}^{\ep_{i+1}}$.
We reduce $R$ to the form
$$
  R = W \bar R.
$$

Let $x,y \in \Var(W)$ and $\si = (x^\ep \mapsto x^\ep y^\de) \in Aut(F_{A \cup V \cup Z})$
be a Nielsen automorphism related to $R$ (i.e.\ $(x^\ep y^{-\de})^{\pm1}$ occurs in $W$).
Application of $\si$ does not change $\bar R$ whenever ~$x$ is cancelled 
or both $x$ and $y$ are degenerate.
Observing that two cancelled variables cannot ``cross'' in $W$ and using Nielsen automorphisms 
of the above form we transform $W$ to a product of commutators 
$$
  [y_1, y_2] \dots [y_{2r-1}, y_{2r}]
$$
where at least one variable $y_j$ or $y_{j+1}$ is degenerate in each commutator $[y_i,y_{i+1}]$.

Finally, we apply a generalization to $\th$ by assigning a single parameter letter 
to each cancelled variable ~$x$.
After the transformation, all of $R$ is written as a product of commutators and we get 
$$
  |x^\th| \le 1 \quad\text{for all } x \in \Var(R) \sm \Var(\bar R)
$$

{\em Step\/} 4: {\em Eliminating $z_1$.} We do this in the same way as in the proof of Theorem \ref{thm:sol-bound}.
The resulting bound increases by $6 c(Q) + 8 m$ for $|x^\th|$ when $x \in \Var(R)$ and by
$3 c(Q) + 4m$ for $|z_i^\th|$.

We have ``almost'' produced the desired
generalization $\th$ of the initial parametric solution ~$\eta$,
with the only difference that $\th$ is formally defined 
on a larger set of variables $V$. This issue is solved by the following observation.

\begin{lemma}
Let $Q \in F_{A \cup X}$ be a standard quadratic word with at least one coefficient. 
If $\phi: F_{A\cup \Var(Q)} \to F_{A\cup V}$ is a primitive $F_A$-monomorphism and $Q^\phi$
is conjugate to $Q$ then $\im\phi = F_{A\cup \Var(Q)}$.
\end{lemma}

\begin{proof}
It follows from Definition \ref{def:primitive-mons} that a primitive $F_A$-monomorphism
$\phi: F_{A\cup \Var(Q)} \to F_{A\cup V}$ can be extended to an $F_A$-automorphism of 
$F_{A\cup V}$. Assume that $\phi\in \Aut_{F_A} (F_{A\cup V})$ and $Q^\phi$ is conjugate to $Q$.
We apply a slightly modified version of Proposition I.4.24 in \cite{Lyndon-Schupp} to
one cyclic word $u_1 = Q$ and the tuple of non-cyclic words $u_i$, $i\ge 2$, consisting of
all letters of $A$. 
Observe that there is no elementary Whitehead automorphism $\rho \in \Aut(F_{A\cup V})$ 
which does not increase the length of all $u_i$ and strictly decreases it for at least one $i$.
Then ~$\phi$ can be represented as a product $\rho_1 \rho_2 \dots \rho_r$ 
of elementary Whitehead automorphisms ~$\rho_i$ which do not change the length of each $u_i$.
We can eliminate permutations and exponent sign changes of generators, so
each $\rho_i$ can be assumed to be an elementary Whitehead $F_A$-automorphism of $F_{A\cup V}$.
An easy inductive argument shows that then for each $i$, the image $Q^{\rho_1\dots\rho_i}$ 
(viewed as a cyclic word) is a quadratic
word equivalent to $Q$ and $(F_{A\cup \Var(Q)})^{\rho_i} = F_{A\cup \Var(Q)}$.
This implies that $(F_{A\cup \Var(Q)})^{\phi} = F_{A\cup \Var(Q)}$.
\end{proof}

The lemma shows that as long as we have a commutative diagram \eqref{eq:param-diagram} where 
$\phi \in \Stab(Q)$ then 
then we can restrict it to $F_{A\cup \Var(Q)}$:
$$
  \xymatrix{
    F_{A\cup \Var(Q)} \ar[r]^{\ti\phi} \ar[d]_\eta &
    F_{A \cup \Var(Q)} \ar[d]^{\ti\th} \\
    F_{A\cup T} &
    F_{A\cup T} \ar[l]_\om
  }
$$
This means precisely that $\ti\th$ is a generalization of $\eta$. 
The proof of Theorem \ref{thm:param-sol-bound} is finished in the case of a 
standard orientable quadratic equation $Q=1$.

The case when $Q$ is semi-standard non-orientable is treated in a similar way with the 
following changes:
\begin{itemize}
\item 
The word $R^*$ under transformation has the form 
$$
  R^* = 
  R^{-1} z_{\si(1)}^{-1} c_{\si(1)}^{\ep_1} z_{\si(1)} \dots z_{\si(m)}^{-1} c_{\si(m)}^{\ep_{m}} z_{\si(m)}.
$$
\item
After Step 1, the equation gets the form
$$
  R = z_1^{-1} c_1^{\ep_1} z_1 \dots z_m^{-1} c_m^{\ep_m} z_m
$$
for some $\ep_1,\dots,\ep_m = \pm1$.
\item
In Step 2, we reduce $R$ to a semi-standard form.
\item
In Step 3, we reduce $W$ to a semi-standard form 
$$
  y_1^2 \dots y_k^2 [y_{k+1}, y_{k+2}] \dots [y_{r-1},y_r]
$$
where each $y_i$ in $y_i^2$ is degenerate and either $y_i$ or $y_{i+1}$ is degenerate in 
each commutator $[y_i, y_{i+1}]$. 
After the reduction, we move $W$ to the right of the last square factor in ~$\bar R$ so that
the whole word $R$ would be written in a semi-standard form.
\item
Before Step 4, we apply Lemma \ref{lm:semistandard-transition} to get $R^* = Q$.
\end{itemize}

We prove the theorem in the general case. Let
$Q=1$ be any quadratic equation in $F_A$.

Let $\phi \in \Aut_{F_A}(F_{A\cup \Var(Q)})$ be an automorphism given by 
Propositions ~\ref{prop:standard-red-bound-orientable} or ~\ref{prop:semistandard-red-bound-nonorientable} 
such that $Q^\phi$ is conjugate to a standard or semi-standard
quadratic word $R$ equivalent to $Q$. 
Then $\phi$ defines a one-two-one correspondence $\eta \mapsto \phi\eta$ between 
parametric solutions $\eta$ of the equation $(R=1,\Var(Q))$ with formal set of variables $\Var(Q)$ and
parametric solutions $\phi\eta$ of the equation $Q=1$. 

Recall that a parametric solution $\eta_1: F_{A\cup \Var(E)} \to F_{A\cup T}$ of $E=1$ is
a generalization of another parametric solution $\eta_2$ of the same equation $E=1$ 
if there are an automorphism $\psi \in\Stab(E)$ and an endomorphism $\om \in \End_{F_A} (F_{A \cup T})$
such that $\eta_2 = \psi\eta_1 \om$. We extend this definition to parametric solutions of equations
$(E=1,V)$ with formal set of variables by taking instead of $\Stab(E)$ the group
$$
  \Stab(E,V) = \setof{\psi \in \Aut_{F_A} (F_{A \cup V})}{E^\psi \text{ is conjugate to } E}.
$$
Clearly, the correspondence $\eta \mapsto \phi\eta$ preserves the relation ``$\eta_1$ is 
a generalization of $\eta_2$'' in the new extended version. 

For a parametric solution $\th$ of $R=1$, let $\hat\th$ denote the parametric solution of $(R=1,\Var(Q)$
defined by extending $\th$ in the following natural way: for each variable $x \in \Var(Q)\sm\Var(R)$ we choose
a parameter letter $t$ (which does not occur in parametric words $y^\th$ for $y \in \Var(R)$)
and set $x^{\hat\th} = t$. Clearly, any parametric solution of $(R=1,\Var(Q))$ has a generalization of
the form ~$\hat\th$ for some ~$\th$. 
It is also obvious that the correspondence $\th \mapsto \hat\th$ 
preserves the relation ``$\th_1$ is a generalization of $\th_2$''.

Now let $\eta$ be an arbitrary parametric solution of $Q=1$. To find the required 
generalization ~$\eta_1$ of $\eta$, we
pass on to the parametric solution
$\psi^{-1} \xi$ of $(R=1,\Var(Q))$ and take its generalization of the form ~$\hat\th$ for some $\th$. 
By the statement of the theorem for the equation $R=1$, there is a generalization $\th_1$ of $\th$
such that for any $x \in \Var(R)$
$$
  |x^{\th_1}| \le 
  \begin{cases}
   8 (c(R) + 2 n(R)) & \text{if $Q$ is orientable,} \\
   36 n(R) (c(R) + 2 n(R)) & \text{if $Q$ is non-orientable.}
  \end{cases}
$$
We take $\eta_1 = \psi\hat\th_1$. By the bounds on $\psi$ in 
Propositions 
~\ref{prop:standard-red-bound-orientable} or ~\ref{prop:semistandard-red-bound-nonorientable}
(and using inequalities $n(R)\le n(Q)$ and $c(R)\le c(Q)$),
for any $x \in \Var(Q)$ we have
\begin{align*}
  |x^{\eta_1}| &\le 4 \sum_{x \in \Var(Q)} |x^{\hat\th_1}| + 2 c(Q) \\
	      & \le
  \begin{cases}
   34 n(R) (c(R) + 2 n(R)) & \text{if $Q$ is orientable,} \\
   146 n(R)^2 (c(R) + 2 n(R)) & \text{if $Q$ is non-orientable.}
  \end{cases}
\end{align*}
as required.
\end{proof}


\end{document}